\documentclass[a4paper,12pt,reqno]{amsart}
\usepackage[latin1]{inputenc}
\usepackage[T1]{fontenc}
\usepackage{amsmath,amssymb}
\usepackage{amsthm}
\usepackage{mathrsfs}
\usepackage{mathtools} 
\usepackage{fancyhdr}
\usepackage{graphicx}
\usepackage{esint}
\usepackage{verbatim}
\usepackage[english]{babel}
\usepackage[a4paper,top=3cm,bottom=2.5cm,left=2cm,right=2cm]{geometry}
\usepackage{yfonts}
\usepackage{hyperref}

\theoremstyle{definition}
\newtheorem{definition}{Definition}[section]
\newtheorem{rmk}[definition]{Remark}

\theoremstyle{plain}
\newtheorem{lemma}[definition]{Lemma}
\newtheorem{theorem}[definition]{Theorem}
\newtheorem{proposition}[definition]{Proposition}
\newtheorem{corollary}[definition]{Corollary}

\newcommand{\res}
{\mathop{\hbox{\vrule height 7pt width .5pt depth 0pt \vrule
height .5pt width 6pt depth 0pt}}\nolimits}

\newcommand{\diam}{\mathrm{diam\,}}

\newcommand{\Lip}{\mathrm{Lip}}
\newcommand{\R}{\mathbb R}

\renewcommand{\subset}{\subseteq}

\newcommand{\propHD}{$\mathcal{D}$}
\newcommand{\preciso}{{\mathfrak p}}

\begin{document}

\title[Fine properties of BV functions in CC spaces]{Fine  properties  of  functions  with  bounded  variation  in  Carnot-Carath\'eodory spaces}

\date{\today}

\author[Don]{Sebastiano Don}
\address[Don and Vittone]{Universit\`a di Padova, Dipartimento di Matematica ``T.~Levi-Civita'', via Trieste 63, 35121 Padova, Italy}
\email{don@math.unipd.it}

\author[Vittone]{Davide Vittone}
\email{vittone@math.unipd.it}



\thanks{The authors are supported by the University of Padova Project Networking and STARS Project ``Sub-Riemannian Geometry and Geometric Measure Theory Issues:  Old and New'' (SUGGESTION), and by  GNAMPA  of  INdAM  (Italy)  project  ``Campi  vettoriali,  superfici  e  perimetri  in  geometrie singolari''. The second named author wishes to ackowledge the support and hospitality of FBK-CIRM (Trento), where part of this paper was written.}

\begin{abstract}
We study properties of functions with bounded variation in Carnot-Ca\-ra\-th\'eo\-do\-ry spaces. We prove their almost everywhere approximate differentiability and we examine their approximate discontinuity set and the decomposition of their distributional derivatives. Under an additional assumption on the space, called property $\mathcal R$, we show that almost all approximate discontinuities are of jump type and we study a representation formula for the jump part of the derivative.
\end{abstract}
\maketitle

\section{Introduction}
A lot of effort was devoted in the last decades to the development of Analysis and Geometry in general metric spaces and, in particular, to the study of functions with  bounded variation ($BV$). {\em Carnot-Carath\'eodory (CC) spaces} are among the most fruitful settings where $BV$ functions have been introduced (\cite{CapDanGar94,FSSC3}), see also \cite{BiroliMosco,DanGarNhiAnnSNS98,FraGalWhe94,FSSC1,FSSC2,GaroNhiCPAM} and the more recent \cite{AGM,AmbMag,AmbScienza,BraMirPal,ComiMag,DMV,Magnani,Marchi,SongYang03}. The aim of this paper is to give some contributions to this research lines by establishing ``fine'' properties of $BV$ functions in CC spaces. A non-trivial part of our work consists in fixing  the appropriate language in a consistent and robust manner.

A CC space is the space $\R^n$ endowed with the  Carnot-Carath\'eodory  distance $d$ (see \eqref{def:dcc}) arising from a fixed family $X=(X_1,\dots,X_m)$ of smooth, linearly independent vector fields (called {\em horizontal}) in $\R^n$ satisfying the H\"ormander condition, see \eqref{eq:CH}. As customary in the literature, we  always assume that metric balls are bounded with respect to the Euclidean topology. Moreover, we  work in {\em equiregular} CC spaces, where a  {\em homogeneous dimension} $Q$, usually larger than the topological dimension $n$, can be defined; recall that any CC space can be  lifted to an equiregular  one, see e.g. \cite{RotStein}.  

The space $BV_X$ of function with {\em bounded $X$-variation} consists of those functions $u$ whose derivatives $X_1u,\dots,X_mu$ in the sense of distributions are represented by a vector-valued measure $D_Xu$ with finite total variation $|D_Xu|$.  These functions have been extensively studied in the literature and important properties have been proved, like coarea formulae, approximation theorems, Poincar\'e inequalities. 

We now describe some of the results we prove in this paper. The first one, Theorem \ref{teo:CalderonZ} below, concerns the almost everywhere {\em approximate $X$-differentiability} (see Section \ref{sec:nozioniapprossimate}) of $BV_X$ functions; its classical counterpart is very well-known, see e.g. \cite[Theorem 3.83]{AFP}. As customary, we denote by $D^a_Xu$ and $D^s_Xu$, respectively, the absolutely continuous  and singular part of $D_Xu$ with respect to the Lebesgue measure $\mathscr L^n$.

\begin{theorem}\label{teo:CalderonZ}
	Let $(\R^n,X)$ be an equiregular CC space, let $\Omega\subseteq \R^n$ be an open set and let $u\in BV_X(\Omega;\R^k)$. Then $u$ is approximately $X$-differentiable at $\mathscr L^n$-almost every point of $\Omega$. Moreover, the approximate $X$-gradient  coincides $\mathscr L^n$-almost everywhere with the density  of  $D_X^au$ with respect to $\mathscr L^n$.
\end{theorem}

The proof of Theorem \ref{teo:CalderonZ}  is based on Lemma \ref{lemma3.81}, that is, on a suitable extension to CC spaces of the inequality
\[
\int_{B(p,r)} \frac{|u(q)-u(p)|}{|q- p|}\:d\mathscr L^n(q)\leq C\int_0^1\frac{|Du|(B(p,tr))}{t^n}\:dt
\]
valid for a classical $BV$ function $u$ on $\R^n$. Lemma \ref{lemma3.81} answers  an open problem stated in \cite{AmbMag} and it is new even in {\em Carnot groups}. We only recall that Carnot groups are connected, simply connected and nilpotent Lie groups whose  Lie algebra is stratified, and we refer to  \cite{FollandStein,Monti,Magnani,NoteLeDonne} for more detailed introduction to the subject. Carnot groups possess a canonical CC structure obtained by fixing a basis $X_1,\dots,X_m$ of the first layer of the Lie algebra of left-invariant vector fields; their importance in the theory stems from the fact that they constitute the infinitesimal models of equiregular CC spaces, a fact that we heavily use  in this paper.

Theorem \ref{teo:CalderonZ} was proved in the setting of  Carnot groups in \cite{AmbMag} together with the following result, which we also extend to our more general setting. We denote by $\mathscr H^{Q-1}$ the Hausdorff measure of dimension $Q-1$ and by $\mathcal S_u$ the set of points where a function $u$ does not possess an approximate limit in the sense of Definition \ref{def:limiteapprossimato}.

\begin{theorem}\label{teo:Suquasirettificabile}
	Let $(\R^n,X)$ be an equiregular CC space, let $\Omega\subseteq \R^n$ be an open set and let $u\in BV_X(\Omega;\R^k)$. Then $\mathcal S_u$ is contained in a countable union of sets with finite $\mathscr H^{Q-1}$ measure.
\end{theorem}

In the classical theory, an important object associated with a $BV$ function $u$ is its jump set: roughly speaking, this is the set of points $p$ for which there exist $u^+(p)\neq u^-(p)$ and a unit direction $\nu_u(p)$ such that, for small $r>0$,  $u$ is approximately equal to $u^+(p)$ on half of $B(p,r)$ and  to $u^-(p)$ on the complementary half of $B(p,r)$, the two halves being separated by an hyperplane orthogonal to $\nu_u(p)$. In this paper we introduce the notion of  {\em approximate $X$-jumps}, see Definition \ref{approximatejump}: this requires a certain amount of preliminary work, expecially about ``fine'' local properties of hypersurfaces with intrinsic $C^1$ regularity ($C^1_X$).

We denote by $\mathcal J_u\subset\mathcal S_u$ the set of $X$-jump points of $u$ and by $(u^+(p),u^-(p),\nu_u(p))$ the approximate $X$-jump triple (see Definition \ref{approximatejump}) at a point $p\in\mathcal J_u$. The measures
\[
D^j_Xu\coloneqq D^s_Xu\res\mathcal J_u,\qquad D^c_Xu\coloneqq D^s_Xu\res(\Omega\setminus\mathcal J_u),
\]
are called, respectively, \emph{jump part}  and \emph{Cantor part} of $D_Xu$. We want to study some further properties of $D_Xu$ and its decomposition 
\[
D_Xu=D_X^au+D_X^su=D_X^au+D_X^cu+D_X^ju.
\]
We  state some of them in the following result, which is a consequence of Theorem \ref{teo:proprieta,a,c,j} and  Proposition \ref{prop3.76}.

\begin{theorem}\label{teo:decomp-intro-1}
Let $(\R^n,X)$ be an equiregular CC space and consider   an open set $\Omega\subseteq \R^n$, a function $u\in BV_X(\Omega;\R^k)$ and a Borel set $B\subseteq \Omega$. Then the following facts hold:
\begin{itemize}
\item[(i)] there exists $\lambda:\R^n\to(0,+\infty)$ (not depending on $\Omega$ nor $u$) locally bounded away from 0 such that $|D_Xu|\geq \lambda|u^+-u^-| \mathscr S^{Q-1}\res \mathcal J_u$;
\item[(ii)] if $\mathscr H^{Q-1}(B)=0$, then $|D_Xu|(B)=0$;
\item[(iii)] if $\mathscr H^{Q-1}(B)<+\infty \text{ and } B\cap\mathcal{S}_u=\emptyset$, then $|D_Xu|(B)=0$;
\item[(iv)] $D^a_Xu=D_Xu\res(\Omega\setminus S)$ and $D^s_Xu=D_Xu\res S$, where 
\[
S\coloneqq\left\{p\in \Omega: \lim_{r\to 0}\frac{|D_Xu|(B(p,r))}{r^Q}=+\infty\right\};
\]
\item[(v)] $\mathcal J_u\subseteq \Theta_u$, where $\Theta_u\subset S$ is defined by
\[
\Theta_u\coloneqq\left\{p\in \Omega: \liminf_{r\to 0}\frac{|D_Xu|(B(p,r))}{r^{Q-1}}>0\right\}.
\]
\end{itemize}
\end{theorem}

However, for classical $BV$ functions much stronger results than Theorems \ref{teo:CalderonZ} and \ref{teo:decomp-intro-1} are indeed known: some of them are proved in the present paper also for $BV_X$ functions under the additional assumption that the space $(\R^n,X)$ satisfies the following natural condition.

\begin{definition}[Property $\mathcal{R}$]\label{rectifiability}
	Let $(\R^n,X)$ be an equiregular CC space with homogeneous dimension $Q$. We say that $(\R^n,X)$ satisfies the {\em property $\mathcal R$} if, for every open set $\Omega\subset\R^n$ and every $E\subseteq \R^n$ with locally finite $X$-perimeter in $\Omega$, the essential boundary $\partial^\ast E\cap\Omega$ of $E$ is countably $X$-rectifiable, i.e., there exists a countable family $(S_i)_{i\in \mathbb N}$ of $C^1_X$ hypersurfaces such that $\mathscr H^{Q-1}(\partial^\ast E\cap\Omega\setminus\cup_{i\in\mathbb N}S_i)=0$. 
\end{definition}

Recall that a measurable set $E\subset\R^n$ has locally finite $X$-perimeter in $\Omega$ if its characteristic function $\chi_E$ has locally bounded $X$-variation in $\Omega$, while we refer to Definition \ref{def:essentialboundary} for the essential boundary $\partial^\ast E$. It was proved in the fundamental paper \cite{Ambrosio01} that the $X$-perimeter measure $|D_X\chi_E|$ of $E$ can be represented as $\theta\mathscr H^{Q-1}\res\partial^\ast E$ for a suitable positive function $\theta$ that is locally bounded away from 0, see Theorem \ref{teo:ambrosio}. 

The validity of property $\mathcal R$ (``rectifiability'') for general equiregular CC spaces is  an interesting open question even  in  Carnot groups (see \cite{AKLD} for a partial result). However,  property $\mathcal R$ is satisfied, besides in Euclidean spaces (\cite{Deg55}), in several interesting situations like Heisenberg groups \cite{FSSC1}, Carnot groups of step 2 \cite{FSSC2} and Carnot groups of {\em  type $\star$} \cite{Marchi}: in particular, Theorems \ref{federervolpert}, \ref{teo:decomp-intro-2propR} and \ref{teo:jumppartHDintro} below hold is such classes. We conjecture that property $\mathcal R$ holds also in all CC spaces of step 2, see \cite{AGM}. Building on the results of \cite{DLDMV}, we prove in Section \ref{sec:proprietaLR} the validity of the weaker {\em property $\mathcal {LR}$} (``Lipschitz rectifiability'', see Definition \ref{def:proprietaLR}) in all Carnot groups satisfying the  algebraic property \eqref{eq:defhyva} below; in particular, a weaker version of Theorem \ref{federervolpert}  holds in such groups, see  Theorem \ref{teo:SuLR}.

The first result we are able to prove assuming property $\mathcal R$ is a refinement of Theorem \ref{teo:Suquasirettificabile} and, roughly speaking, it states that $\mathscr H^{Q-1}$-almost all singularities of a $BV_X$ function  are of jump type. 

\begin{theorem}\label{federervolpert}
	Let $(\R^n,X)$ be an equiregular CC space satisfying property $\mathcal R$, let $\Omega\subset\R^n$ be an open set and let $u\in BV_X(\Omega; \R^k)$. Then $\mathcal S_u$ is countably $X$-rectifiable and $\mathscr H^{Q-1}(\mathcal S_u\setminus\mathcal J_u)=0.$
\end{theorem}

Assuming property $\mathcal R$, also Theorem \ref{teo:decomp-intro-1} can be refined as follows.

\begin{theorem}\label{teo:decomp-intro-2propR}
Under the assumption and notation of Theorem \ref{teo:decomp-intro-1}, assume that $(\R^n,X)$ satisfies property $\mathcal R$. Then
\begin{itemize}
\item[(i)]  $\mathscr H^{Q-1}(\Theta_u\setminus \mathcal J_u)=0$ and $D^j_Xu=D_Xu\res\Theta_u$;
\item[(ii)] $D_X^cu=D_Xu\res(S\setminus \Theta_u)$;
\item[(iii)] if  $B\subset\Omega$ is such that  $\mathscr H^{Q-1}\res B$ is $\sigma$-finite, then $D_X^cu(B)=D_X^au(B)=0$.
\end{itemize}	
\end{theorem}

Theorem \ref{teo:decomp-intro-2propR} is part of  Theorem \ref{teo:proprieta,a,c,j}. We also mention that, assuming property $\mathcal R$, one can define a {\em precise representative} $u^\preciso$ of $u$ (see \eqref{eq:defrapprespreciso}) and prove that the convergence of the mean values $\fint_{B(p,r)}u\:d\mathscr L^n$ to $u^\preciso(p)$ holds, as $r\to 0$, for $\mathscr H^{Q-1}$-almost every $p$. See Theorem \ref{teo:convergenzaalrapprpreciso}.

Eventually, a further natural assumption,  {\em property} \propHD\ (``density'', see Definition \ref{def:propHD}), concerning the local behavior of the spherical Hausdorff measure $\mathscr S^{Q-1}$ of $C^1_X$ hypersurfaces, allows  to obtain  a stronger result about the jump part $D^ju$, see Theorem \ref{teo:jumppartHDintro}. Property \propHD\ is satisfied in Riemannian manifolds (trivially),   Heisenberg groups, Carnot groups of step 2 and Carnot groups of   type $\star$, see section \ref{sec:proprietaLR}; its validity in more general settings is an interesting open problem that will be object of future investigations. Theorem \ref{teo:jumppartHDintro} follows from the more general Theorem \ref{teo:jumppartHDcompleto}, which deals with a representation of the restriction of $D_Xu$ to any  countably $X$-rectifiable set $R$.

\begin{theorem}\label{teo:jumppartHDintro}
	Let $(\R^n,X)$ be an equiregular CC space satisfying properties $\mathcal R$ and \propHD; then, there exists a function $\sigma:\R^n\times\mathbb S^{m-1}\to(0,+\infty)$ such that, for every open set $\Omega\subset\R^n$ and every  $u\in BV_X(\Omega; \R^k)$,  one has
\[
D_X^ju=\sigma(\cdot,\nu_u)(u^+-u^-)\otimes \nu_u\: \mathscr S^{Q-1}\res \mathcal J_u.
\]
\end{theorem}

The paper is structured as follows. In Section \ref{sec:preliminari} we introduce the preliminary material about CC spaces and their nilpotent approximation (Section \ref{sec:spaziCC}), $C^1_X$ hypersurfaces and $X$-rectifiable sets (Section \ref{sec:C1X}), approximate $X$-jumps and $X$-differentiability (Section \ref{sec:nozioniapprossimate}) and $BV_X$ functions (Section \ref{sec:BVX}). Most of the material in Sections \ref{sec:C1X} and \ref{sec:nozioniapprossimate} is original.  Section \ref{sec:dimostrazioni} contains the proof of our results, while in Section \ref{sec:proprietaLR} we discuss some classes of Carnot groups satisfying properties $\mathcal R,\ \mathcal {LR}$ and/or \propHD. Eventually, we collected in Appendix \ref{app:A} some useful result from Geometric Measure Theory in metric spaces and in Appendix \ref{app:risultatitecnici} the proofs of some (new but) technical results (Borel regularity,  etc.) about the approximate $X$-jump and the approximate $X$-differentiability sets.\medskip

{\em Acknowledgements.} It is a pleasure to thank V. Magnani, R. Monti, D. Morbidelli and D. Pallara for their interest in this paper and for several stimulating discussions.

\section{Preliminaries}\label{sec:preliminari}
\subsection{Carnot-Carath\'eodory spaces and nilpotent approximation}\label{sec:spaziCC}
	In what follows $\Omega$ will denote an open set in $\R^n$ and $X=(X_1,\dots,X_m)$ an $m$-tuple ($m\leq n$) of smooth and  linearly independent vector fields on $\R^n$, with $2\leq m\leq n$. We say that an absolutely continuous curve $\gamma:[0,T]\rightarrow\R^n$ is an \emph{$X$-subunit path} joining $p$ and $q$ if $\gamma(0)=p$, $\gamma(T)=q$ and there exist $h_1,\dots,h_m\in L^\infty([0,T];\R)$ such that $\sum_{j=1}^m h_j^2\leq 1$ and for almost every $t \in [0,T]$ one has
	\[
        \dot{\gamma}(t)=\sum_{j=1}^m h(t)X_j(\gamma(t)).
	\]
For every $p,q\in \R^n$, we define the quantity
	\begin{equation}\label{def:dcc}
	d(p,q)\coloneqq\inf\left\{T>0: \exists\text{ a $X$-subunit path } \gamma \text{ joining } p \text{ and } q\right\},
	\end{equation}
where we agree that $\inf\emptyset=+\infty$.

	A sufficient condition that makes $d$ a metric on $\R^n$ is the following
\begin{theorem}[Chow-Rashevsky]\label{Chow}
Suppose that  
\begin{equation}\label{eq:CH}
\forall\ p\in\R^n\qquad\mathcal{L}ie\{X_1,\dots,X_m\}(p)=T_p\R^n\cong\R^n,
\end{equation}
where $\mathcal{L}ie\{X_1,\dots,X_m\}(p)$ denotes the linear span of all  iterated commutators of the vector fields $X_1,\dots,X_m$ computed at $p$. Then $d$ is a distance. 
\end{theorem}

We will refer to \eqref{eq:CH} as  {\em H\"ormander condition}. When \eqref{eq:CH} holds, the couple $(\R^n,X)$ is said to be a \emph{Carnot-Carath\'eodory space} of {\em rank} $m$. We denote by $B(p,r)$ the $d$-ball of center $p\in\R^n$ and radius $r>0$.

For every $p\in \R^n$ and for every $i\in \mathbb N$ we denote by $\mathcal L^i(p)$ the linear span of all the commutators of $X_1,\dots,X_m$ up to order $i$ computed at $p$. Notice that $\mathcal{L}ie\{X_1,\dots,X_m\}(p)=\bigcup_{i\in\mathbb N}\mathcal L^i(p)$. We say that $(\R^n, X)$ is \emph{equiregular} if there exist natural numbers $n_0,n_1,\dots,n_s$ such that
\[
0=n_0< n_1<\cdots< n_s=n\qquad\text{and}\qquad\forall\ p\in\R^n\ \dim \mathcal L^i(p)=n_i.
\]
The natural number $s$ is called \emph{step} of the Carnot-Carath\'eodory space.

In the following theorem we resume some well-known facts about the geometry of  equi\-re\-gu\-lar CC spaces, see e.g.  \cite{NSW,Mitchell}. Recall that a Radon measure $\mu$ on a metric space $(M,d)$ is {\em doubling} if there exists $C>0$ such that 
\[
\mu(B(x,2r))\leq C\mu(x,r)\qquad\text{for every $x\in M$ and  every $r>0$}.
\]

\begin{theorem}\label{ccproperties}
Let $(\R^n,X)$ be an equiregular $CC$ space of step $s$. Then the following facts hold.
	\begin{itemize}
		\item[(i)] For every compact set $K\subseteq\R^n$ there exists $M\geq 1$ such that
\[
\frac{1}{M}\:|p-q|\:\leq d(p,q)\leq M|p-q|^{\frac{1}{s}}\qquad\text{for any }p,q\in K.
\]
\item[(ii)] The Hausdorff dimension of the metric space $(\R^n,d)$ is $Q\coloneqq\sum_{i=1}^si(n_i-n_{i-1})$.
\item[(iii)] The metric measure space $(\R^n,d,\mathscr L^n)$ is locally Ahlfors $Q$-regular, i.e., for every compact set $K\subseteq \R^n$ there exist $R>0$ and $C>1$ such that for every $p\in K$ and for every $r\in (0,R)$
\begin{equation}\label{ahlfors}
\frac{1}{C}r^Q\leq \mathscr L^n(B(p,r))\leq Cr^Q.
\end{equation}
In particular,  $(\R^n, d, \mathscr L^n)$ is locally doubling.
\end{itemize}
\end{theorem}

As customary, we assume from now on that the metric balls $B(p,r)$ are bounded with respect to the Euclidean metric in $\R^n$; this implies that the CC space $(\R^n, X)$ is {\em geodesic}, i.e., that for every $p,q\in \R^n$ there exists a $X$-subunit curve realizing the infimum in \eqref{def:dcc}. The existence of length minimizing curves  implies that, for every $p\in \R^n$ and for every $r>0$, one has $\mathscr L^n (\partial B(p,r))=0$; see Proposition \ref{stella}.

\begin{definition}[Adapted exponential coordinates]\label{def:flagexp}
  Let $(\R^n,X)$ be an equiregular CC spa\-ce and let $p\in\R^n$ be fixed; choose an open neighborhood $V\subset\R^n$ of $p$ and smooth vector fields $Y_1,\dots, Y_n$ such that
\begin{itemize}
\item $Y_i=X_i$ for any $i=1,\dots,m$;
\item for every $k=1,\dots,s$ the vector fields $Y_{n_{k-1}+1},\dots,Y_{n_k}$ are chosen among the $k$-order commutators of $X_1,\dots,X_m$;
\item for every $q\in V$ and every  $k=1,\dots, s$ the set $\{Y_1(q),\dots,Y_{n_k}(q)\}$ is a basis of $\mathcal L^k(q)$.
\end{itemize}
Then there exists a neighborhood $U$ of $0$ in $\R^n$ for which the map 
	\begin{equation}\label{eq:defcoord}
	\begin{split}
	F: U&\rightarrow \R^n\\
	x&\mapsto \exp(x_1Y_1+\dots+x_nY_n)(p)
	\end{split}
	\end{equation}
is well defined. We say that $(x_1,\dots ,x_n)$ are \emph{adapted exponential coordinates} around $p$.
\end{definition}

The definition of $F$ depends on the point $p$; when confusion may arise, we underline this dependence by using the notation $F_p$ to denote (for any $x\in\R^n$ for which it is defined) the map $F_p(x)\coloneqq\exp(x_1Y_1+\dots+x_nY_n)(p)$. When needed, we will also write $F(p,x)$ to denote  $\exp(x_1Y_1+\dots+x_nY_n)(p)$; notice that, for every bounded set $V\subset\R^n$, one can find an open neighborhood $U$ of $0$ in $\R^n$ such that $F$ is well defined in $V\times U$.

For every  $p\in \R^n$ and every $j=1,\dots,m$  we define
	\[
	\widetilde X_j\coloneqq dF_p^{-1}(X_j\circ F_p).
	\] 
	It is readily seen that if $X$ satisfies the H\"ormander condition, then also $\widetilde X$ does and we denote by $\widetilde d$ the CC distance in (a suitable open subset of) $\R^n$ associated with the $m$-tuple of vector fields $\widetilde X=(\widetilde X_1,\dots,\widetilde X_m)$, and by $\widetilde B(x,r)$ the metric balls  associated with $\widetilde d$. Again, when confusion may arise we shall use the notation $\widetilde B_p(x,r)$ to specify that the metric ball is induced by the map $F_p$. Since $dF_p(0)e_j=Y_j(p)$, we have $\widetilde X_j(0)=e_j$ for every $j=1,\dots,m$. Moreover it is easy to verify that for every $p\in \R^n$ and every sufficiently small $r>0$ one has
	\[
d(F_p(x_1),F_p(x_2))=\widetilde d(x_1,x_2)\qquad \forall\ x_1, x_2 \in \widetilde B(0,r);
	\]
in particular, $F_p(\widetilde B(x,r))=B(F_p(x),r)$.
        
        \begin{rmk}
Let us consider $\mu_p\coloneqq(F^{-1}_p)_\#\mathscr{L}^n$, i.e., the measure defined for every Borel set $A$ in $\mathbb R^n$ by
	\[
	\mu_p(A)=\mathscr L^n\left(F_p(A)\right)=\int_A\left|\det\nabla F_p\right|d\mathscr L^n.
	\]
It is easy to see that, whenever $0<\varepsilon<|\det\nabla F_p(0)|$, there exists an open neighborhood $U$ of $0$ such that
	\begin{equation}\label{measureequivalence}
	  \left(\left|\det\nabla F_p(0)\right|-\varepsilon\right)\mathscr L^n \res U\leq \mu_p\res U\leq \left(\left|\det\nabla F_p(0)\right|+\varepsilon\right)\mathscr L^n \res U.
	\end{equation}
        \end{rmk}

\begin{definition}[Degree, dilations and pseudo-norm]
  If $(\R^n,X)$ is an equiregular CC space and $p,Y_1,\dots,Y_n$ are  as in Definition \ref{def:flagexp}, we define the {\em degree} $w_j$ of the coordinate $j$ by $Y_j(p)\in \mathcal L^{w_j}(p)\setminus \mathcal L^{w_{j-1}}(p)$ or, equivalently, by $n_{w_{j-1}}<j\leq n_{w_j}$. For every $r>0$, the anisotropic \emph{dilation} $\delta_r:\R^n\to\R^n$ is defined by 
\begin{equation}\label{dilations}
\delta_r(x)\coloneqq\left(x_1,\dots,r^{w_i}x_i,\dots,r^sx_n\right).
\end{equation}
We  say that a function $f:\R^n\rightarrow \R$ is {\em $\delta$-homogeneous} of degree $w\in \mathbb N$ if for every $p\in \R^n$ and every $\lambda>0$ one has $f(\delta_\lambda p)=\lambda^wf(p)$. We also introduce the pseudo-norm 
\[
\|x\|\coloneqq\sum_{j=1}^n|x_j|^{1/w_j},\qquad x\in\R^n
\]
and the pseudo-balls
\begin{equation}\label{eq:defAr}
A(r)\coloneqq\left\{x\in\R^n:\|x\|\leq r\right\}.
\end{equation}
Clearly, $\delta_r\left(A(1)\right)=A(r)$.
\end{definition}

The following result is proved in \cite{NSW}.
\begin{theorem}\label{normestimate}
	Let $K\subseteq \R^n$ be a compact set in an equiregular CC space $(\R^n,X)$ and let $U$ be a neighborhood of $0$ such that, for every $p\in K$, the map $F_p$ is well-defined in $U$. Then there exists $C>1$ such that for every $x\in U$ and every $p\in K$ we have
	\[
	\frac 1C\|x\|\leq \widetilde d_p(0,x)\leq C\|x\|.
	\] 
\end{theorem}

The following theorem is classical, see e.g. \cite{Bellaiche1} or \cite{MPV}. For an introduction to {\em Carnot groups} (also known as {\em stratified groups}) see for instance \cite{FollandStein,Monti,Magnani,NoteLeDonne}.

\begin{theorem}\label{th:tangentspace}
  Let $(\R^n,X)$ be an equiregular CC space and let $p\in \R^n$ be fixed; then, there exists a family   $\widehat X\coloneqq(\widehat X_1,\dots, \widehat X_m)$ of polynomial vector fields in $\R^n$ such that
  \begin{itemize}
  \item[(i)] for every $j=1,\dots,m$, $\widehat X_j$ is 1-homogeneous, i.e. $(d\delta_r)[\widehat X_j]=r\widehat X_j\circ\delta_r$ for all $r>0$;
  \item[(ii)] for every $j=1,\dots,m$ we have $r(d\delta_{r^{-1}})[\widetilde X_j\circ\delta_r]\to \widehat X_j$ in $C^\infty_{loc}(\R^n)$;
  \item[(iii)] the couple $(\R^n, \widehat X)$ is associated with a Carnot group structure on $\R^n$;
  \item[(iv)] (\cite[Remark 2.6]{MPV}) $\widehat X$ can be completed to a basis $\widehat X_1,\dots, \widehat X_n$ of the Lie algebra of the Carnot group in such a way that $x=\exp(\sum_{j=1}^nx_j\widehat{X}_j)(0)$ for any $x\in\R^n$.
  \end{itemize}
\end{theorem}

The vector fields $\widehat X_1,\dots, \widehat X_m$ introduced in Theorem \ref{th:tangentspace} are known in the literature as the \emph{nilpotent approximation} of $X_1,\dots, X_m$ at the point $p$; we will say that the structure $(\R^n, \widehat X)$ is \emph{tangent} to $(\R^n,X)$ at $p$. We shall denote by $\widehat d$ the Carnot-Carath\'eodory distance associated with $\widehat X$ and by $\widehat B$ the corresponding balls; recall that $\widehat d(\delta_r x,\delta_r y)=r\widehat d(x,y)$ for any $r>0$ and $x,y\in\R^n$. When confusion may arise, we shall use the notation $\widehat B_p,\widehat d_p$ to specify the dependence on the point $p$. 

By the Carnot group structure there exists $\widehat C=\widehat C_p>0$ such that
\begin{equation}\label{eq:pallacarnot}
\mathscr L^n(\widehat B_p(x,r))=\widehat C r^Q\qquad \forall\ x\in \R^n,\ r>0.
\end{equation}
The constant $\widehat C$ depends on $p$; however, given a compact set $K\subset \R^n$, there exists $M>0$ such that $1/M\leq \widehat C_p\leq M$ for any $p\in K$. See Remark \ref{rmk:ccappuccio} below.

We will need later the following simple result.

\begin{proposition}\label{symmetry}
	Let $(\R^n,X)$ be an equiregular CC space, and let $r>0$. Then for every $p\in\R^n$ one has
	\[
	 x\in\widehat B_p(0,r)\Longleftrightarrow -x\in \widehat B_p(0,r).
	\]
\end{proposition}
\begin{proof}
By well-known properties of Carnot groups and Theorem  \ref{th:tangentspace} (iv)  we  have
\[
-x=\exp\left(-\sum_{j=1}^nx_j\widehat{X}_j\right)(0)=\left[\exp\left(\sum_{j=1}^nx_j\widehat X_j\right)(0)\right]^{-1}=x^{-1},
\]
which combined with the left invariance of $\widehat d$ with respect to the group operation implies
\[
\widehat d(0,-x)=\widehat{d}(0,x^{-1})=\widehat d(x\cdot 0,x\cdot x^{-1})=\widehat d(x,0).
\]
This concludes the proof.
\end{proof}

We  recall for future references the following well-known result, for which we refer e.g. to \cite{Bellaiche1,Mitchell}.
 
\begin{theorem} \label{th:localgroup}
	Let $(\R^n,X)$ be an equiregular CC space and let $p\in\R^n$ be fixed; then
	\begin{equation}\label{eq:localgroup}
		\lim_{r\to 0} \left(\sup\left\{\frac{|\widetilde d_p(x,y)-\widehat d_p(x,y)|}{r}: x,y\in \widetilde B_p(0,r)\right\}\right)=0.
	\end{equation}
In particular, for any $\varepsilon >0$, there exists $R>0$ such that 
\[
\widehat B_p(0,(1-\varepsilon)r)\subseteq\widetilde B_p\left(0,r\right)\subseteq \widehat B_p(0,(1+\varepsilon)r)\qquad\text{for any }r\in(0,R).
\]
\end{theorem}

\begin{rmk}\label{rmk:ccappuccio}
Let $K\subset\R^n$ be a compact set; then there exists $M\geq 1$ such that the constant $\widehat C=\widehat C_p$ appearing in \eqref{eq:pallacarnot} satisfies
  \[
  \frac 1M\leq \widehat C_p\leq M\qquad\forall\ p\in K.
  \]
This follows because, by Theorem \ref{th:localgroup}, for any $p\in K$
  \[
  \begin{aligned}
  \widehat C_p&=\lim_{r\to 0}\frac{\mathscr L^n(\widehat B_p(0,r))}{r^Q}=\lim_{r\to 0}\frac{\mathscr L^n(\widetilde B_p(0,r))}{r^Q}=\lim_{r\to 0}\frac{\mathscr L^n(F_p^{-1}(B(p,r)))}{r^Q}\\&=\frac 1{|\det \nabla F_p(0)|}\lim_{r\to 0}\frac{\mathscr L^n(B(p,r))}{r^Q}
  \end{aligned}
  \]
and one can conclude by using Theorem \ref{ccproperties} (iii) and the smoothness of $F(p,x)$.  
\end{rmk}

\subsection{Hypersurfaces of class \texorpdfstring{$C_X^1$}{C1X}}\label{sec:C1X}
This section is devoted to the study of hypersurfaces with intrinsic $C^1$ regularity; we work in a fixed equiregular CC space $(\R^n,X)$. As customary, given an open set $\Omega\subset\R^n$ we denote by $C^1_X(\Omega)$ the space of continuous functions $f:\Omega\to\R$ such that the derivatives $X_1f,\dots,X_mf$ are represented, in the sense of distributions, by continuous functions.

\begin{definition}[Hypersurface of class $C_X^1$]
	We say that $S\subseteq \R^n$ is a {\em $C_X^1$ hypersurface} if for every $p\in S$ there exist $R>0$ and $f\in C_X^1(B(p,R))$ such that the following facts hold
	\begin{itemize}
		\item[(i)] $S\cap B(p,R)=\{q\in B(p,R): f(q)=0\};$
		\item[(ii)] $Xf\neq 0$ on $ B(p,R)$.
	\end{itemize}
In this case, for every $p$ in $S$ we define the {\em horizontal normal} $\nu_S(p)\in\mathbb S^{m-1}$ to $S$ at $p$ letting 
	\[
	\nu_S(p)\coloneqq\frac{Xf(p)}{|Xf(p)|}.
	\] 
\end{definition}

The horizontal normal is well-defined up to a sign and, in particular, it does not depend on the choice of $f$: this is a consequence, for instance, of Corollary \ref{th:hypersurface}, below.

We will also use the notion of intrinsic Lipschitz regularity for hypersurfaces introduced in \cite{Vittone2}.  In the next definition, the Lipschitz continuity of $f$ is understood with respect to the CC distance; recall that $f:\Omega\to\R$ is locally Lipschitz on an open set $\Omega\subset\R^n$ if and only if it is continuous and its distributional derivatives $X_1f,\dots,X_m f$ belong to $L^\infty_{loc}(\Omega)$; see \cite{FSSCBUMI,GNJAM}. 

\begin{definition}[$X$-Lipschitz hypersurface]
We say that $S\subseteq \R^n$ is an {\em $X$-Lipschitz hypersurface} if for every $p\in S$ there exist $R>0$ and a Lipschitz map $f: B(p,R)\rightarrow \R$ such that
  \begin{itemize}
  \item[(i)] $B(p,R)\cap S=\{q\in B(p,R): f(q)=0\}$;
  \item[(ii)] there exist $C>0$ and $1\leq j\leq m$ such that $X_jf\geq C$  $\mathscr L^n$-a.e. on $B(p,R)$.
  \end{itemize}
\end{definition}

Hypersurfaces with $X$-Lipschitz or $C^1_X$ regularity have locally finite $(Q-1)$-dimensional Hausdorff measure, see \cite{Vittone2}.

Given $\nu\in\R^m$ we define $\widetilde L_\nu: \R^n\rightarrow \R$ letting
\begin{equation}\label{def:elletilde}
\widetilde L_\nu (x)\coloneqq \sum_{i=1}^m \nu_i x_i.
\end{equation}
This notation will be extensively used throughout the paper. The following proposition shows that the maps $\widetilde L_\nu$ provide a sort of first-order ``linear'' approximation for $C^1_X$ functions.

\begin{proposition}\label{prop7.2}
	Let $p\in\R^n$, $R>0$ and $f\in C_X^1(B(p,R))$ be fixed; then
	\[
	\lim_{r\to 0}\left(\sup\left\{\frac{|f(F_p(x))-f(p)-\widetilde L_{Xf(p)}(x)|}{r}: x \in \widetilde B(0,r)\right\}\right)=0.
	\] 
\end{proposition}
\begin{proof}
	It is not restrictive to assume that $f(p)=0$. Let $r\leq R$ and take $x\in \widetilde B(0,r)$. Set $d\coloneqq\widetilde d(x,0)$ and take a geodesic $\gamma\in \Lip([0,d];\R^n)$ such that $\gamma(0)=0$, $\gamma(d)=x$ and there exists $h:[0,d]\rightarrow \mathbb R^m$ such  that for $\mathscr L^1$-a.e. $t\in [0,d]$ we have 
	\[
	|h(t)|=1 \qquad\text{and}\qquad \dot{\gamma}(t)=\sum_{j=1}^mh_j(t)\widetilde X_j(\gamma(t)).
	\] 
	Notice that $\widetilde X_j(0)=e_j$, hence there exists $C>0$ such that $|\widetilde X_j(y)-e_j|\leq Cr$ for every $y\in \widetilde B(0,r)$ and  every $j=1,\dots, m$. 	Therefore, for every $k=1,\dots,m$
	\[
	\begin{aligned}
	  \left|x_k-\int_0^d h_k(t)dt\right| &= \left|\left(\int_0^d \dot\gamma(t)dt\right)_k-\int_0^d h_k(t)dt\right|\\
          &=\left|\sum_{j=1}^m\int_0^dh_j(t)\left(\widetilde X_j(\gamma(t))\right)_kdt-\sum_{j=1}^m\int_0^d h_j(t)\left(e_j\right)_kdt\right|\\
          &=\left|\sum_{j=1}^m\int_0^dh_j(t)\left(\widetilde X_j(\gamma(t))-e_j\right)_kdt\right|\leq mCrd\leq mCr^2.
	\end{aligned}
	\]
	Hence, if for every $x\in \widetilde B(0,r)$ we set $d\coloneqq\widetilde d(x,0)$ and we denote by $h$ a control associated with the geodesic $\gamma$ joining $0$ and $x$, we have
	\begin{equation}\label{eq:control}
		\lim_{r\to 0} \left(\sup\left\{\frac{1}{r}\left|x_k-\int_0^dh_k(t)dt\right|: x\in \widetilde B(0,r), k=1,\dots,m\right\}\right)=0.
	\end{equation}
	Notice also that for every $x\in \widetilde B(0,r)$
	\[
	\begin{split}
		f(F_p(x))&=f(F_p(x))-f(F_p(0))\\
		&=f(F_p(\gamma(d)))-f(F_p(\gamma(0)))=\int_0^d\sum_{j=1}^mX_jf(F_p(\gamma(t)))h_j(t)dt.
	\end{split} 
	\]
	Let $\varepsilon>0$ be fixed. By \eqref{eq:control} and the continuity of $Xf$  we can choose $r_0\in(0,R)$ such that
\[
\begin{array}{ll}
\sup\left\{\dfrac 1r\left|(x_1,\dots,x_m)-\displaystyle\int_0^dh(t)dt\right|: x\in \widetilde B(0,r)\right\}<\dfrac{\varepsilon}{2|Xf(p)|}\quad & \forall\ r\in(0,r_0)\vspace{.2cm}\\
|Xf(F_p(x))-Xf(p)|<\dfrac{\varepsilon}{2} & \forall\ x\in \widetilde B(0,r_0).
\end{array}
\]
For any $r\in(0,r_0)$ and $x\in \widetilde B(0,r)$ we have
	\[
	\begin{split}
		&|f(F_p(x))-\widetilde L_{Xf(p)}(x)|\\
		=\:&\left|\int_0^d\left\langle h(t),Xf(F_p(\gamma(t)))\right\rangle dt -\sum_{j=1}^m X_jf(p)x_j\right|\\
		\leq\:&\int_0^d|h(t)| |Xf(F_p(\gamma(t)))-Xf(p)|dt +|Xf(p)|\left|(x_1,\dots,x_m)-\int_0^dh(t)dt\right|\\
		<\:&d\frac{\varepsilon}{2} + |Xf(p)|\left|(x_1,\dots,x_m)-\int_0^dh(t)dt\right|.
	\end{split}
	\]
	 The result follows dividing both sides by $r$ and taking into account that $d\leq r$.
\end{proof}

An immediate consequence of Proposition \ref{prop7.2} is Corollary \ref{th:hypersurface}, where we start using the following very convenient notation: given $t\in\R$ and a function $f:I\to\R$ defined on some set $I$, we denote by $\{f>t\},\{f=t\}$, etc. the sets $\{x\in I:f(x)>t\},\{x\in I:f(x)=t\}$, etc. This notation will be extensively used in the paper.

\begin{corollary} \label{th:hypersurface}
  Let  $p\in \R^n$ and $f\in C_X^1(B(p,R))$ for some $R>0$; suppose that $f(p)=0$, $Xf\neq 0$ in $B(p,R)$ and consider the $C_X^1$ hypersurface  $S\coloneqq\left\{q\in B(p,R): f(q)=0\right\}$. Then, for every $\varepsilon>0$, there exists $r_0>0$ such that, for every $r\in(0,r_0)$
        \begin{equation}\label{eq:striscia}
	F^{-1}_p(S)\cap \widetilde B(0,r)\subseteq \{x\in \widetilde B(0,r): -\varepsilon r\leq \widetilde L_{Xf(p)}(x)\leq\varepsilon r\}.
	\end{equation}
        Moreover 
        \begin{equation}\label{eq:discordanza}
	\lim_{r\to 0}\frac{\mathscr L^n(\{x\in \widetilde B(0,r): f(F_p(x))\widetilde L_{Xf(p)}(x) <0\})}{r^Q}=0.
	\end{equation}
\end{corollary}
\begin{proof}
	Fix $\varepsilon>0$ and apply Proposition \ref{prop7.2} to get $r_0>0$ such that for every $0<r<r_0$ and for every $x\in \widetilde B(0,r)$ we have $|f(F_p(x))-\widetilde L_{Xf(p)}(x)|\leq \varepsilon r$.
	Then, if we take $x\in \widetilde B(0,r)\cap \{\widetilde L_{Xf(p)}\geq 2\varepsilon r\}$, we also get
	\[
	f(F_p(x))\geq\varepsilon r.
	\]
        Reasoning in the same way with the set $\{\widetilde L_{Xf(p)}\leq -2 \varepsilon r\}$ we readily get \eqref{eq:striscia}. The previous argument shows that for any $\varepsilon >0$ there exists $r_0>0$ such that for any $r\in (0,r_0)$ we have
        \[
        \widetilde B(0,r)\cap\{(f\circ F_p)\widetilde L_{Xf(p)} \leq 0\}\subseteq  \widetilde B(0,r)\cap\{ -\varepsilon r\leq \widetilde L_{Xf(p)}\leq\varepsilon r\}.
        \]
        The proof of \eqref{eq:discordanza} follows by noticing that, by Theorem \ref{normestimate}
        \[
\mathscr L^n( \widetilde B(0,r)\cap\{ -\varepsilon r\leq \widetilde L_{Xf(p)}\leq\varepsilon r\})\leq C\varepsilon r^Q,
\]
for a suitable constant $C$ independent of $r$.
\end{proof}

We point out  for future references the following observation.

\begin{rmk}\label{rmk:germ}
  Let $(\R^n,X)$ be an equiregular CC space, $p\in \R^n$, $R>0$ and suppose that $f_1,f_2 \in C_X^1(B(p,R))$ are such that $f_1(p)=f_2(p)=0$ and $Xf_1(p)=Xf_2(p)$; then one has
  \[
  \lim_{r\to 0} \frac{1}{r^Q}\mathscr L^n( B(p,r)\cap\{ f_1f_2\leq 0\})=0.
  \]
Indeed, taking into account \eqref{eq:striscia} we observe that
  \[
  \lim_{r\to 0} \frac{1}{r^Q}\mathscr L^n(\{\xi \in B(p,r): f_1(\xi)f_2(\xi)= 0\})=0.
  \]
  On the other hand, since $\widetilde L_{Xf_1(p)}=\widetilde L_{Xf_2(p)}$ the set $B(p,r)\cap\{ f_1f_2<0\}$ is contained in 
  \[
  B(p,r)\cap\Big(  \{f_1f_2<0 \text{ and } \widetilde L_{Xf_1(p)}\circ F_p^{-1}>0\}\cup \{f_1f_2<0 \text{ and } \widetilde L_{Xf_1(p)}\circ F_p^{-1}\leq 0\}\Big),
  \]
  that combined with \eqref{eq:discordanza} completes the proof.
\end{rmk}

We can now introduce the notion of intrinsic rectifiability in equiregular CC spaces. We denote by $\mathscr H^k$ and $\mathscr S^k$, respectively, the  $k$-dimensional Hausdorff and spherical Hausdorff measures in $(\R^n,d)$, see e.g. Definition \ref{def:Hausmeas}.

\begin{definition}[$X$-rectifiability]
	Let $(\R^n,X)$ be an equiregular CC space of homogeneous dimension $Q\in \mathbb N$ and let $R\subseteq \R^n$. We say that $R$ is {\em countably $X$-rectifiable} (respectively,  {\em countably $X$-Lipschitz rectifiable}) if there exists a family $\{S_h:h\in \mathbb N\}$ of $C_X^1$ hypersurfaces (resp., $X$-Lipschitz hypersurfaces) such that
	\begin{equation}\label{eq:defrett}
	\mathscr H^{Q-1}\left(R\setminus \bigcup_{h=0}^\infty S_h\right)=0.
	\end{equation}
	Moreover we say that $R$ is  {\em $X$-rectifiable} (resp.,  {\em $X$-Lipschitz rectifiable}) if $R$ is countably $X$-rectifiable (resp., countably $X$-Lipschitz rectifiable) and $\mathscr H^{Q-1}(R)<+\infty$.
\end{definition}

\begin{definition}[Horizontal normal]\label{def:normalerettif}
Let $R\subset \R^n$ be countably $X$-rectifiable and let $(S_h)$ be  $C_X^1$ hypersurfaces such that \eqref{eq:defrett} holds. Then the {\em horizontal normal} $\nu_R:R\to \mathbb S^{m-1}$ to $R$ is defined by
\[
\nu_R(p)\coloneqq\nu_{S_h}(p)\qquad\text{if }p\in R\cap S_h\setminus\bigcup_{k<h}S_k\,.
\]
\end{definition}

The horizontal normal $\nu_R$ is well-defined, up to a sign, $\mathscr H^{Q-1}$-almost everywhere on $R$: this is a standard consequence of the following result. 

\begin{proposition}
	Let $(\R^n,X)$ be an equiregular CC space and let $S_1,S_2\subseteq\R^n$ be two hypersurfaces of class $C_X^1$. Then the set
	\[
	E\coloneqq \left\{p\in S_1\cap S_2: \nu_{S_1}(p)\notin\{ \pm \nu_{S_2}(p)\} \right\}
	\]
is $\mathscr H^{Q-1}$-negligible.
\end{proposition}
\begin{proof}
  By a localization argument we can suppose without loss of generality that $S_1$ is bounded in $\R^n$ and that $\mathscr H^{Q-1}(E)\leq \mathscr H^{Q-1}(S_1)<+\infty$. For every $\delta>0$ define 
	\[
	E_{\delta}\coloneqq\{p\in E: |\langle\nu_{S_1}(p),\nu_{S_2}(p)\rangle|\leq 1-\delta\}.
	\]
	Then we have $ E=\bigcup\{E_{\delta}:\delta \in (0,+\infty)\cap\mathbb Q\}$.
	
	Fix $\varepsilon\in (0,1/4)$ and define for every $R>0$ the set $E_{\delta,R}$ of all the points $p$ of $E_\delta$ such that the following three properties hold for every $r\leq 2R$
	\begin{itemize}
		\item[(a)] if $C>0$ is the constant appearing in Theorem \ref{normestimate}, for every $x\in A(Cr)$ we have $\widehat B_p(x,\varepsilon r)\subseteq \widetilde B_p(x,2\varepsilon r)$;
		\item[(b)] for $i=1,2$ we have $F_p^{-1}(S_i\cap B(p,2r))\subseteq\{|\widetilde L_{\nu_{S_i}(p)}|<\varepsilon r\}$;
		\item[(c)] $\diam B(p,r)=\diam \widetilde B_p(0,r)\geq r$.
	\end{itemize}
 By Theorems \ref{th:localgroup} and \ref{th:hypersurface} and the fact\footnote{This is an easy consequence of the fact that the curve $t\mapsto \exp(t\widehat X_1)$ is globally length minimizing. This fact is well-known to experts, even though the only reference we are aware of is \cite[Proposizione 7.4]{LaureaVittone}.} that $\diam \widehat B_p(0,r)=2r$  we deduce that $E_{\delta,R}\nearrow E_\delta$ as $R\to 0$.
 
	Fix now $\eta\in(0,\frac R2)$. Then there exist a sequence $(q_h)$ in $\R^n$ and a sequence $(r_h)$ in $(0,\eta)$ such that
	\[
	E_{\delta,R}\subseteq \bigcup_{h=0}^\infty B(q_h,r_h) \quad\text{ and }
	\]
	\[
	\sum_{h=0}^\infty (r_h)^{Q-1}\leq\sum_{h=0}^\infty\left(\diam B(q_h,r_h)\right)^{Q-1}\leq\mathscr S^{Q-1}_\eta(E_{\delta,R})+1.
	\]
	We can suppose without loss of generality that for every $h\in \mathbb N$ there exists $p_h\in B(q_h,r_h)\cap E_{\delta,R}$. Therefore for every $h\in \mathbb N$ one has $B(q_h,r_h)\subseteq B(p_h,2r_h)$ and consequently
	\[
	E_{\delta,R}\subseteq\bigcup_{h=0}^\infty B(p_h,2r_h).
	\]
	Taking into account Theorem \ref{normestimate}, we can find $C>0$ such that for every $h\in \mathbb N$ one has 
	\[
	F_{p_h}^{-1}(E_{\delta,R}\cap  B(p_h,2r_h))\subseteq A_h
	\]
where
\[
A_h\coloneqq\left\{x\in\R^n:\|x\|\leq C r_h \text{ and } |\widetilde L_{\nu_{S_i}(p_h)}(x) |\leq\varepsilon r_h, \text{ for } i=1,2 \right\}.
	\] 
	We prove now that $\mathscr L^n(A_h)\leq C_\delta \varepsilon^2 r_h^Q$ for some $C_\delta>0$ depending on $\delta$. In fact, since $|\langle\nu_{S_1}(p_h),\nu_{S_2}(p_h)\rangle|\leq 1-\delta$, we have (up to an orthogonal change of  coordinates)\
	\[
	\left\{x\in \R^n: |\widetilde L_{\nu_{S_i}(p_h)}(x)|<\varepsilon r_h \text{ for } i=1,2 \right\}\subseteq Q^2(0,C_{\delta}\varepsilon r_h)\times\R^{n-2},
	\]
where the notation $Q^2(z,s)$ denotes a $2$-dimensional cube of center $z$ and size $s$. Hence 
	\[
	A_h\subseteq \left(Q^2(0,C_\delta\varepsilon r_h)\cap \left\{x\in \R^m:\sum_{j=1}^m|x_j|\leq Cr_h\right\} \right) \times \left\{ x\in \R^{n-m}:\sum_{j=m+1}^n|x_j|^{\frac{1}{d_j}}\leq Cr_h \right\}
	\]
and consequently $\mathscr L^n(A_h)\leq C_\delta \varepsilon^2 r_h^Q$. 
        
        For every $h\in \mathbb N$, combining Theorem \ref{5rcovering} and the fact that $A_h$ is compact, we can find $N_h\in \mathbb N$ and a family $\{x_{h,j}:j=1,\dots,N_h\}$ of points of $A_h$ such that $\{\widehat B_{p_h}(x_{h,j},\varepsilon r_h):j=1,\dots,N_h\}$ covers $A_h$ and $\{\widehat B_{p_h}(x_{h,j},\varepsilon \tfrac{r_h}{5}):j=1,\dots,N_h\}$ is pairwise disjoint.
 Reasoning as above, it is easy to see that 
	\[
	\mathscr L^n\left(\left\{x\in \R^n:\widehat d_{p_h}(x,A_h)<\frac{\varepsilon r_h}{5}\right\}\right)\leq \widetilde C_\delta\varepsilon^2 r_h^Q.
	\]
	Therefore we can estimate
	\[
	N_h\leq \frac{\mathscr L^n\left(\left\{x\in \R^n:\widehat d_{p_h}(x,A_h)<\frac{\varepsilon r_h}{5}\right\}\right)}{\mathscr L^n\left(\widehat B_{p_h}(x_{h,j},\frac{\varepsilon r_h}{5})\right)}\leq \widehat C_\delta \varepsilon^{2-Q}
	\]
for some $\widehat C_\delta>0$ that, by Remark \ref{rmk:ccappuccio}, depends only on $\delta$. By property (a) we have also $\widehat B_{p_h}(x_{h,j},\varepsilon r_h)\subseteq \widetilde B_{p_h}(x_{h,j},2\varepsilon r_h)$,  hence the family $\{\widetilde B_{p_h}(x_{h,j},2\varepsilon r_h): j=1,\dots, N_h\}$ is a covering of $A_h$, that is also a covering of $F_{p_h}^{-1}(E_{\delta,R}\cap B(p_h,r_h))$. Hence the family $\{B(F^{-1}_{p_h}(x_{h,j}),2\varepsilon r_h):j\in \mathbb N\}$ is a covering of $E_{\delta , R}\cap B(p_h,2r_h)$
	In particular, since $\varepsilon \in (0,1/4)$ we have 
	\[
	\begin{aligned}
	  \mathscr S^{Q-1}_{\eta}(E_{\delta, R})&\leq \mathscr S^{Q-1}_{4\varepsilon \eta} (E_{\delta, R})\leq \sum_{h=0}^\infty\mathscr S^{Q-1}_{4\varepsilon \eta}(E_{\delta, R}\cap B(p_h, 2r_h))\\
          & \leq\sum_{h=0}^\infty \sum_{j=1}^{N_h} \left(\diam B(F_{p_h}^{-1}(x_{h,j}), 2\varepsilon r_h)\right)^{Q-1}\leq \sum_{h=0}^\infty N_h(4\varepsilon r_h)^{Q-1}\\
          &\leq \sum_{h=0}^\infty \widehat C_\delta\varepsilon r_h^{Q-1}\leq \widehat C_\delta\varepsilon (\mathscr S_\eta ^{Q-1}(E_{\delta, R})+1).
	\end{aligned}
	\]
	Letting $\eta \to 0$ we get $\mathscr S^{Q-1}(E_{\delta, R})\leq \widehat C_\delta\varepsilon(\mathscr S^{Q-1}(E_{\delta, R})+1)$, which gives, letting $R\to 0$
        \[
        \mathscr S^{Q-1}(E_\delta)\leq \widehat C_\delta\varepsilon(\mathscr S^{Q-1}(E_\delta)+1).
        \]
        Letting now $\varepsilon \to 0$ we get, for any $\delta >0$, that $\mathscr S^{Q-1}(E_\delta) =0$ , i.e., $\mathscr S^{Q-1}(E)=0$. This concludes the proof.
\end{proof}

\subsection{Approximate notions of continuity, \texorpdfstring{$X$}{X}-jumps and  \texorpdfstring{$X$}{X}-dif\-fe\-ren\-tia\-bi\-li\-ty}\label{sec:nozioniapprossimate}
In this section we introduce the notions of approximate continuity, approximate $X$-jumps and approximate $X$-differentiability; we keep on working in a fixed equiregular CC space $(\R^n,X)$. We use the notation 
\[
\fint_A u\,d\mathscr L^n\coloneqq\frac{1}{\mathscr L^n(A)}\int_{A} u\,d\mathscr L^n
\]
and, in what follows, we denote by $\Omega$  an open subset of $\R^n$.

\begin{definition}[Approximate Limit]\label{def:limiteapprossimato}
	Let  $u\in L^1_{loc}(\Omega;\R^k)$, $z\in \R^k$ and $p\in \Omega$. We say that $z\in \R^k$ is the {\em approximate limit} of $u$ at $p$ if
\[
\lim_{r\to 0}\fint_{B(p,r)}|u-z|d\mathscr L^n=0.
\]
We  denote by $u^\star(p)$ the approximate limit of $u$ at $p$ and by $\mathcal S_u$ the set of points in $\Omega$ where $u$ does not admit an approximate limit.
\end{definition}

If the approximate limit of $u$ at a point $p$ exists, it is also unique. By the generalized Lebesgue's differentiation theorem (see e.g. \cite[Section 2.7]{Heinonen}), we have $\mathscr L^n(\mathcal S_u)=0$ and $u^\star=u$ a.e. on $\Omega$. Moreover it can be easily proved (adapting e.g. \cite[Proposition 3.64]{AFP}) that $\mathcal S_u$ is a Borel set and that $ u^\star:\Omega\setminus\mathcal S_{u}\rightarrow \R^k$ is a Borel map.

\begin{rmk}\label{rem:applimIFFconvL1}
Let $\Omega,u,z$ and $p$ be as in Definition \ref{def:limiteapprossimato}. Then $u$ has approximate limit $z$ at $p$ if and only if, working in adapted exponential coordinates $F_p$ around $p$, as $r\to 0$ the functions $u\circ F_p\circ \delta_r$ converge in $L^1_{loc}(\R^n;\R^k)$ to the constant function $z$. This is an easy exercise left to the reader; alternatively, it is enough to follow the proof of Proposition \ref{jumpequiv} below with $a=b=z$. 
\end{rmk}

\begin{definition}[Essential boundary]\label{def:essentialboundary}
  Given a measurable set $E\subseteq \R^n$ and $t\in[0,1]$ we denote by $E^t$ the set of points with density $t$ for $E$, i.e., the set of all $p\in \R^n$ satisfying
  \[
\lim_{r\to 0}\frac{\mathscr L^n(E\cap B(p,r))}{\mathscr L^n(B(p,r))}=t.
\]
The {\em essential boundary} of $E$ is $\partial^*E\coloneqq\R^n\setminus(E^0\cup E^1)$.
  \end{definition}
  
The following proposition is standard; for the reader's convenience we prove it later in Proposition \ref{levelsetsbis}.

\begin{proposition}\label{levelsets}
Let $u\in L^1_{loc}(\Omega)$, $p\in \Omega\setminus \mathcal S_u$ and $t\neq u^\star(p)$; then $p\notin \partial^*\{u>t\}$.  
\end{proposition}

We now introduce the notion of $X$-jump points; this requires a certain  amount of work, one of the reasons being that there is no canonical way of separating a CC ball $B(p,r)$ into complementary ``half-balls'' $B^+_\nu(p,r),B^-_\nu(p,r)$. We will use as separating sets an arbitrary hypersurface $S$ of class $C^1_X$ such that $\nu_S(p)=\nu$, and one of the issues (Remark \ref{jumpwellposed} below) is proving well-posedness of our definition independently of the choice of $S$.

For any fixed $p\in \R^n$, $\nu\in \mathbb S^{m-1}$ and $r>0$ we introduce the notation $B^+_\nu(p,r)$ and $B^-_\nu(p,r)$ as follows. Given $R>0$ and $f\in C_X^1(B(p,R))$ such that\footnote{One can consider for instance $f=\widetilde L_\nu\circ F_p$.} $f(p)=0$ and $Xf (p)/|Xf(p)|=\nu$, we set for $r\in(0,R)$
\begin{equation}\label{eq:Bpmnu}
  B^+_\nu(p,r)\coloneqq   B(p,r)\cap\{ f>0\}
\qquad\text{and}\qquad
  B^-_\nu(p,r)\coloneqq   B(p,r)\cap\{ f<0\}.
\end{equation}
These objects are well-defined only if $r$ is small enough. Moreover, there is a clear abuse of notation, since $B_\nu^{\pm}(p,r)$ depend on the choice of $f$. However, this will not effect the validity of our results.

Before introducing the notion of approximate $X$-jumps we state some properties of the ``half-balls'' $B_\nu^{\pm}(p,r)$. Proposition \ref{goodsplit} is used in the proof of Theorem \ref{teo:convergenzaalrapprpreciso}.

\begin{proposition}\label{goodsplit}
	Let $(\R^n,X)$ be an equiregular $CC$ space and let $\Omega\subseteq\R^n$ be an open set. Then, for any $p\in\Omega$ and $\nu \in \mathbb S^{m-1}$. 
	\[
	\lim_{r\to 0}\frac{\mathscr L^n\left(B_\nu^+(p,r)\right)}{\mathscr L^n\left(B(p,r)\right)}=\lim_{r\to 0}\frac{\mathscr L^n\left(B_\nu^-(p,r)\right)}{\mathscr L^n\left(B(p,r)\right)}=\frac{1}{2}
	\]
\end{proposition}
\begin{proof}
	Let $U$ be a neighborhood of $p$ and let $f\in C^1_X(U)$ be such that $f(p)=0$ and $Xf(p)=\nu$.
	Choose $\varepsilon\in (0,1)$. By Proposition \ref{prop7.2} and Theorem \ref{th:localgroup} we can suppose without loss of generality that for every small enough $r$ one has $F_p(\widetilde B(0,r))=B(p,r)$ and 
	\begin{equation}\label{limsup}
		F_p^{-1}(B_\nu^+(p,r))=\widetilde B(0,r)\cap\left\{f\circ F_p>0\right\}
	\subseteq \widehat B(0,(1+\varepsilon)r)\cap \{\widetilde L_{\nu}\geq -\varepsilon r\}.
	\end{equation}
	Analogously
	\begin{equation}\label{liminf}
		\widehat B(0,(1-\varepsilon)r)\cap\{\widetilde L_\nu\geq \varepsilon r\}\subseteq\widetilde B(0,r)\cap \{f\circ F_p>0\}
		=F_p^{-1}(B_\nu^+(p,r)).
	\end{equation}
	Applying $\delta_{1/r}$ to both sides of \eqref{limsup} and evaluating the Lebesgue measure we get
	\[
	\begin{aligned}
		\frac{\mathscr L^n\left(F_p^{-1}(B_\nu^+(p,r))\right)}{r^Q}&=\mathscr L^n\left(\delta_{1/r}\left(F_p^{-1}(B_\nu^+(p,r))\right)\right)\\
		&\leq\mathscr L^n\left(\widehat B(0,1+\varepsilon)\cap\{\widetilde L_\nu\geq -\varepsilon\}\right).
	\end{aligned}
	\]
	Taking the $\limsup$ as $r\to 0$ and letting $\varepsilon\to 0$ we infer
	\[
	\limsup_{r\to 0}\frac{\mathscr{L}^n\left(F_p^{-1}(B_\nu^+(p,r))\right)}{r^Q}\leq\mathscr L^n\left(\widehat{B}(0,1)\cap\{\widetilde L_\nu\geq 0\}\right)
	=\frac{1}{2}\mathscr L^n(\widehat B(0,1)),
	\]
	where the last equality follows from Proposition \ref{symmetry}. With the same argument, from \eqref{liminf} we get
	\[
	\liminf_{r\to 0}\frac{\mathscr L^n\left(F_p^{-1}(B_\nu^+(p,r))\right)}{r^Q}\geq \frac{1}{2}\mathscr L^n(\widehat B(0,1)),
	\]
hence
	\begin{equation}\label{lim1}
	\lim_{r\to 0} \frac{\mathscr L^n\left(F_p^{-1}(B_\nu^+(p,r))\right)}{r^Q}=\frac{1}{2}\mathscr L^n(\widehat B(0,1)).
	\end{equation}
By Theorem \ref{th:localgroup}
	\begin{equation}\label{lim2}
	\lim_{r\to 0}\frac{\mathscr L^n(\widetilde B(0,r))}{r^Q}=\lim_{r\to 0}\mathscr L^n(\delta_{1/r}(\widetilde B(0,r)))=\mathscr L^n(\widehat B(0,1)),
	\end{equation}
	and combining \eqref{lim1} and \eqref{lim2} we get
	\[
	\lim_{r\to 0} \frac{\mathscr L^n\left(F_p^{-1}\left(B_\nu^+(p,r)\right)\right)}{\mathscr L^n(\widetilde B(0,r))}=\frac 12.
	\]
	If $c\coloneqq|\det\nabla F(0)|>0,$ using \eqref{measureequivalence} we notice that for every $0<\varepsilon<c$ and every sufficiently small $r>0$ we have 
	\[
	\frac{\left(c-\varepsilon\right)\mathscr L^n\left(F_p^{-1}\left(B^+_\nu(p,r)\right)\right)}{\left(c+\varepsilon\right)\mathscr L^n(\widetilde B(0,r))}\leq\frac{\mathscr L^n\left(B_\nu^+(p,r)\right)}{\mathscr L^n(B(p,r))}\leq\frac{(c+\varepsilon)\mathscr L^n\left(F_p^{-1}\left(B^+_\nu(p,r)\right)\right)}{(c-\varepsilon) \mathscr L^n(\widetilde B(0,r))}.
	\]
	The result follows passing to the limit as $r\to 0$, letting $\varepsilon \to 0$ and, eventually, using a similar argument for $B_\nu^-$.
\end{proof}

We can now introduce the notion of $X$-jump points.

\begin{definition}[Approximate $X$-jumps]\label{approximatejump}
	Let $u\in L^1_{loc}(\Omega;\R^k)$ and $p\in \Omega$. We say that $u$ has an \emph{approximate $X$-jump} at $p$ if there exist $a,b\in \R^k$ with $a\neq b$ and $\nu \in \mathbb{S}^{m-1}$ such that
	\begin{equation}\label{eq:defsalti}
	\lim_{r\to 0}\fint_{B^+_\nu(p,r)}|u-a|d\mathscr L^n=\lim_{r\to 0}\fint_{B^-_\nu(p,r)}|u-b|d\mathscr L^n=0.
	\end{equation}
In this case we say that $(a,b,\nu)$ is an \emph{approximate $X$-jump triple} of $u$ at $p$.
We shall denote by $\mathcal J_u$ the set of approximate $X$-jump points of $u$ and by $(u^+(p),u^-(p),\nu_u(p))$ the  approximate $X$-jump triple for $u$ at $p\in \mathcal J_u$. 
\end{definition}

\begin{rmk}\label{rem:JusubsetSu}
Using e.g. Proposition \ref{goodsplit} one easily proves that $\mathcal J_u\subset\mathcal S_u$.
\end{rmk}

Notice that, if $u$ has an approximate $X$-jump at $p$ associated with $(a,b,\nu)$, then it is also associated with the triple $(b,a,-\nu)$. For this reason, it will be sometimes convenient to consider the space of  triples endowed with the equivalence relation $(a,b,\nu)\equiv (a',b',\nu')$ if and only if $(a,b,\nu)=(a',b',\nu')$ or $(a,b,\nu)=(b',a',-\nu')$. The following Proposition \ref{jumpequiv}  shows that the $X$-jump triple $(u^+(p),u^-(p),\nu_u(p))$ is unique up to equivalence, for the map $\R^k\times\R^k\times\mathbb S^{m-1}\ni(a,b,\nu)\to w_{a,b,\nu}\in L^1_{loc}(\R^n;\R^k)$ defined by \eqref{eq:saltotangente} below satisfies
\[
w_{a,b,\nu}=w_{a',b',\nu'}\quad\Longleftrightarrow\quad (a,b,\nu)\equiv(a',b',\nu').
\]

In the theory of classical $BV$ functions a jump point can be detected, via a blow-up procedure, in terms of $L^1_{loc}$-convergence to a function taking two different values on complementary half-spaces; this is the content of the next statement, which also gives an equivalent definition of approximate $X$-jump points.

\begin{proposition}\label{jumpequiv}
	Let $(\R^n, X)$ be an equiregular $CC$ space, $\Omega$  an open set,  $u\in L^1_{loc}(\Omega;\R^k)$, $p\in \Omega$ and let $a,b\in\R^k$ with $a\neq b$ and $\nu\in\mathbb S^{m-1}$ be fixed. Then the following statements are equivalent:
\begin{itemize}
\item[(i)] $p\in \mathcal J_u$ and $(u^+(p),u^-(p),\nu_u(p))\equiv(a,b,\nu)$;
\item[(ii)]  working in adapted exponential coordinates $F_p$ around $p$, as $r\to 0$ the functions $\widetilde u_r\coloneqq u\circ F_p\circ \delta_r$ converge in $L^1_{loc}(\R^n;\R^k)$ to 
	\begin{equation}\label{eq:saltotangente}
	w_{a,b,\nu}(y)\coloneqq\begin{cases}
	a\quad & \text{if } \widetilde L_{\nu}  (y)>0\\
	b & \text{if } \widetilde L_{\nu} (y)<0.
	\end{cases}
	\end{equation}
	\end{itemize}
\end{proposition}
\begin{proof}
We can assume without loss of generality that $k=1$. 

We prove  the implication (i)$\Rightarrow$(ii); we can assume that $(u^+(p),u^-(p),\nu_u(p))=(a,b,\nu)$ and, writing $w\coloneqq w_{a,b,\nu}$, we prove that for any fixed $R>0$ one has
\[
\lim_{r\to0}\int_{\widehat B(0,R)}|u\circ F_p\circ\delta_r-w|\:d\mathscr L^n=0.
\]
By a change of variables, this is equivalent to proving that
\begin{equation}\label{eq:lc}
\lim_{r\to0}\frac{1}{r^Q}\int_{\widehat B(0,r)}|u\circ F_p-w|\:d\mathscr L^n=0.
\end{equation}
Let $f$ be the real function of class $C^1_X$ defined on a neighborhood of $p$ used to define, as in \eqref{eq:Bpmnu}, the half-balls $B^\pm_\nu(p,r)$ appearing in \eqref{eq:defsalti};  we set for brevity
\begin{align*}
& \widehat B^+_\nu(0,r)\coloneqq \widehat B(0,r)\cap\{\widetilde L_\nu>0\}, &&\widehat B^-_\nu(0,r)\coloneqq\widehat B(0,r)\cap\{\widetilde L_\nu<0\}\\
& \widetilde B^+_\nu(0,r)\coloneqq \widetilde B(0,r)\cap\{ f\circ F_p>0\}, &&\widetilde B^-_\nu(0,r)\coloneqq\widetilde B(0,r)\cap\{ f\circ F_p<0\}.
\end{align*}
By Theorem \ref{th:localgroup} there exists an increasing function $\omega:(0,+\infty)\to(0,+\infty)$ such that 
\[
\lim_{r\to 0}\frac{\omega(r)}r=0\qquad\text{and}\qquad\widehat B(0,r) \subset  \widetilde B(0,r+\omega(r))
\]
for any sufficiently small $r$. Therefore
\begin{align*}
& \frac{1}{r^Q}\int_{\widehat B(0,r)}|u\circ F_p-w|\:d\mathscr L^n\\
= &  \frac{1}{r^Q}\bigg(\int_{\widehat B^+_\nu(0,r)}|u\circ F_p-a|\:d\mathscr L^n+\int_{\widehat B^-_\nu(0,r)}|u\circ F_p-b|\:d\mathscr L^n\bigg)\\
\leq & \frac{1}{r^Q}\bigg(\int_{\widetilde B^+_\nu(0,r+\omega(r))}|u\circ F_p-a|\:d\mathscr L^n
+\int_{\widehat B^+_\nu(0,r)\setminus\widetilde B^+_\nu(0,r+\omega(r))}\big(|u\circ F_p-b|+|a-b|\big)\:d\mathscr L^n\\
& \hphantom{\frac{1}{r^Q}\bigg(}+\int_{\widetilde B^-_\nu(0,r+\omega(r))}|u\circ F_p-b|\:d\mathscr L^n
+\int_{\widehat B^-_\nu(0,r)\setminus\widetilde B^-_\nu(0,r+\omega(r))}\big(|u\circ F_p-a|+|a-b|\big)\:d\mathscr L^n\bigg)
\intertext{and using $\widehat B^\pm_\nu(0,r)\setminus\widetilde B^\pm_\nu(0,r+\omega(r))\subset \widetilde B(0,r+\omega(r))\setminus\widetilde B^\pm_\nu(0,r+\omega(r))\subset\overline{\widetilde B^\mp_\nu(0,r+\omega(r))}$}
\leq & \frac{1}{r^Q}\bigg(2\int_{\widetilde B^+_\nu(0,r+\omega(r))}|u\circ F_p-a|\:d\mathscr L^n + 2\int_{\widetilde B^-_\nu(0,r+\omega(r))}|u\circ F_p-b|\:d\mathscr L^n\\
& \hphantom{\frac{1}{r^Q}\bigg(}+|a-b| \mathscr L^n\big(\widetilde B(0,r+\omega(r))\cap \{(f\circ F_p)\widetilde L_\nu\leq0\}\big)\bigg)
\end{align*}
and \eqref{eq:lc} follows from \eqref{eq:defsalti} and Corollary 	\ref{th:hypersurface} taking also Theorem \ref{ccproperties} into account.

For the converse implication one has to prove that, if (ii) holds and $f$ is a $C^1_X$ real function  on a neighborhood of $p$ such that $f(p)=0$ and $Xf (p)/|Xf(p)|=\nu$, then \eqref{eq:defsalti} holds with  $B^\pm_\nu(p,r)$  defined (see \eqref{eq:Bpmnu}) in terms of $f$. By Theorem \ref{ccproperties} and a change of variables, proving \eqref{eq:defsalti} amounts to proving that
\[
\lim_{r\to0}\frac{1}{r^Q}\int_{\widetilde B^+_\nu(0,r)}|u\circ F_p-a|\:d\mathscr L^n = \lim_{r\to0}\frac{1}{r^Q}\int_{\widetilde B^-_\nu(0,r)}|u\circ F_p-b|\:d\mathscr L^n=0
\]
and this can be done by a boring adaptation, that we omit, of the previous argument.
\end{proof}

\begin{rmk}\label{jumpwellposed}
The proof of Proposition \ref{jumpequiv} implicitly shows that the validity of \eqref{eq:defsalti} does not depend on the choice of the function $f$ used in \eqref{eq:Bpmnu} to define $B^\pm_\nu(p,r)$.
\end{rmk}

The proof of the following result is standard and we postpone it to the Appendix \ref{app:risultatitecnici}.

\begin{proposition}\label{jumpprop}
Let $(\R^n,X)$ be an equiregular $CC$ space, $\Omega$ be an open set and let $u\in L^1_{loc}(\Omega;\R^k)$. Then the following facts hold: 
	\begin{itemize}
	\item[(i)] $\mathcal{J}_u$ is a Borel set and, up to a choice of a representative for $X$-jump triples, the function
		\[
		\begin{split}
		\mathcal J_u&\rightarrow\R^k\times\R^k\times\mathbb S^{m-1}\\
		p&\mapsto (u^+(p),u^-(p),\nu_u(p))
		\end{split}
		\]
is Borel;
	\item[(ii)] for every $f\in \Lip(\R^k;\R^h)$ and $p\in \mathcal J_u$ we have 
		\[
		p\in\mathcal{J}_{f\circ u} \quad\Longleftrightarrow\quad f(u^+(p))\neq f(u^-(p))
		\] 
and in this case $\displaystyle((f\circ u)^+(p),(f\circ u)^-(p),\nu_{f\circ u}(p))\equiv(f(u^+(p)),f(u^-(p)),\nu_u(p))$. Otherwise, $p\notin \mathcal S_{f\circ u}$ and $(f\circ u)^\star(p)=f(u^+(p))=f(u^-(p))$.
\end{itemize}
\end{proposition}

We now pass to he introduction of approximate $X$-differentiability.

\begin{definition}[Approximate $X$-differentiability]\label{def:Xdiff}
Let $u\in L^1_{loc}(\Omega;\R^k)$ and $p\in \Omega\setminus \mathcal S_u$. We say that $u$ is \emph{approximately $X$-differentiable} at $p$ if there exist a neighborhood $U$ of $p$ and  $f\in C^1_X(U;\R^k)$ such that $f(p)=0$ and
	\begin{equation}\label{approximatediff}
	\lim_{r\to 0}\fint_{B(p,r)}\frac{|u- u^\star(p)-f|}{r}d\mathscr L^n=0.
	\end{equation}
The subset of points of $\Omega$ in which $u$ is approximately $X$-differentiable is denoted by $\mathcal D_u$. 
\end{definition}

If $f$ is as in Definition \ref{def:Xdiff} we will call $Xf(p)\in\R^{k\times m}$ the \emph{approximate $X$-gradient} of $ u$ at $p$. By the following proposition the approximate $X$-gradient of $u$ at $p$ is uniquely determined, and we denote it by $D^{ap}_X u(p)$.

\begin{proposition}[Uniqueness of approximate $X$-gradient]\label{wellposedness}
  Let  $(\R^n,X)$ be an equi\-re\-gu\-lar CC space, $\Omega\subset\R^n$ an open set, $u\in L^1_{loc}(\Omega;\R^k)$ and $p\in \Omega\setminus \mathcal S_u$. Let $R>0$ and  $f_1,f_2 \in C^1_X(B(p,R);\R^k)$; suppose that formula \eqref{approximatediff} holds for both $f=f_1$ and  $f=f_2$. Then $p\in \mathcal D_u$, $f_1(p)=f_2(p)=0$ and $Xf_1(p)=Xf_2(p)$. 
  
Conversely, if $f_1(p)=f_2(p)=0$ and $Xf_1(p)=Xf_2(p)$, then formula \eqref{approximatediff} holds for $f=f_1$ if and only if it holds for $f=f_2$.
\end{proposition}
\begin{proof}
It is not restrictive to assume that $k=1$.  Define for $i=1,2$ the functions $L_i\coloneqq \widetilde L_{Xf_i(p)}$. Suppose first that both $f_1, f_2$ satisfy \eqref{approximatediff}. Fix $\varepsilon>0$ and by Proposition \ref{prop7.2} choose $r>0$ such that for every $\varrho \in (0,r)$ 
	\[
	\frac{|f_i\circ F_p-L_i|}{\varrho} <\frac{\varepsilon}{2}\qquad\text{on }\widetilde B(0,\varrho).
	\]
	Then for such values of $\varrho$ we have
	\[
	\begin{aligned}
	\fint_{\widetilde B(0,\varrho)}\frac{|L_1-L_2|}{\varrho}d\mathscr L^n & \leq\fint_{\widetilde B(0,\varrho)}\frac{\left|f_1\circ F_p-f_2\circ F_p\right|}{\varrho}d\mathscr L^n+\varepsilon\\
	&\leq C \fint_{B(p,\varrho)} \frac{\left|f_1-f_2\right|}{\varrho}d\mathscr L^n +\varepsilon\\
	&\leq C\fint_{B(p,\varrho)}\frac{\left|u-u^\star(p)-f_1\right|+\left|u- u^\star(p)-f_2\right|}{\varrho}d\mathscr L^n+\varepsilon .
	\end{aligned}
	\]
	It follows that
	\[
	\lim_{\varrho\to 0}\fint_{\widetilde B(0,\varrho)}\frac{\left|L_1-L_2\right|}{\varrho}d\mathscr L^n=0.
	\]
	If $Xf_1(p)\neq Xf_2(p)$, by Theorem \ref{normestimate} one would get, for some $C_1>0$
	\[
	\begin{aligned}
	\fint_{\widetilde B(0,\varrho)}\left|L_1-L_2\right|d\mathscr L^n&=\frac{1}{\mathscr L^n(\widetilde B(0,\varrho))}\int_{\widetilde B(0,\varrho)}\left|L_1-L_2\right|d\mathscr L^n\\
	&\geq \frac{1}{\mathscr L^n\left(A(C_1\varrho)\right)}\int_{A(\varrho/C_1)}|L_1-L_2|d\mathscr L^n= C \frac{\varrho^{Q+1}}{\varrho^Q}=C \varrho,
	\end{aligned}
	\]
 a contradiction. This proves the first part of the statement.

	Suppose now that $Xf_1(p)=Xf_2(p)$ and that $f_1$ satisfies \eqref{approximatediff}. Then we have $L_1=L_2$ and
	\[
	\begin{aligned}
	&\fint_{B(p,\varrho)}\frac{|u-u^\star(p)-f_2|}{\varrho}d\mathscr L^n\\
	 \leq&\fint_{B(p,\varrho)}\frac{|f_1-L_1\circ F_p^{-1}|+|u(y)- u^\star(p)-f_1|+|f_2-L_2\circ F_p^{-1}|}{\varrho}d\mathscr L^n.
\end{aligned}
	\]
	By Proposition \ref{prop7.2} this completes the proof.
\end{proof}

As for $X$-jump points, also approximate $X$-differentiability points can be detected by a blow-up procedure.

\begin{proposition}\label{diffequiv}
	Let $(\R^n, X)$ be an equiregular CC space, $\Omega$ be an open subset of $\R^n$, $u\in L^1_{loc}(\Omega;\R^k)$ and let $p\in \Omega\setminus \mathcal S_u$. Then $u$ is approximate $X$-differentiable at $p$ if and only if there exists $z=(z_1,\dots,z_k)\in \R^{k\times m}$ such that
	\[
	\frac{u\circ F_p\circ \delta_r-u^\star(p)}{r}\to (\widetilde L_{z_1},\dots,\widetilde L_{z_k})\quad\text{ in $L^1_{loc}(\R^n;\R^k)$ as }r\to0.
	\]
In this case we have $D_X^{ap}u(p)=z$.
\end{proposition}
\begin{proof}
We assume without loss of generality that $k=1$. Assume first that $p\in \mathcal D_u$ and let $f$ be as in \eqref{approximatediff}; set $z\coloneqq D_X^{ap}u(p)\in\R^m$. Given $R>0$, by Theorem \ref{th:localgroup} one has for small enough $r$
	\[
	\begin{aligned}
	&\int_{\widehat B(0, R)}\left|\frac{u\circ F_p\circ \delta_r- u^\star(p)}{r}-\widetilde L_z\right|d\mathscr L^n\\
	=&\frac{1}{r^Q}\int_{\widehat B(0,r R)}\left|\frac{u\circ F_p-u^\star(p)-\widetilde L_z}{r}\right|d\mathscr L^n
	\leq  \frac{1}{r^Q}\int_{\widetilde B(0,2r R)}\left|\frac{u\circ F_p-u^\star(p)-\widetilde L_z}{r}\right|d\mathscr L^n\\
	\leq&  \frac{C}{r^Q}\int_{ B(p,2r R)}\left|\frac{u-u^\star(p)-\widetilde L_z\circ F_p^{-1}}{r}\right|d\mathscr L^n
		\to0\quad\text{as }r\to0,
	\end{aligned}
	\]
(we used Proposition \ref{wellposedness}), which proves the first part of the statement.

	Conversely, for any small enough $r>0$ we have
	\[
	\begin{aligned}
&\fint_{B(p,r)}	\left|\frac{u-u^\star(p)-\widetilde L_z\circ F_p^{-1}}{r}\right|d\mathscr L^n\leq 	
	\frac{C}{r^Q}\int_{\widetilde B(0,r)}\left|\frac{u\circ F_p-u^\star(p)-\widetilde L_z}{r}\right|d\mathscr L^n\\
	 \leq & \frac{C}{r^Q}\int_{\widehat B(0,2r)}\left|\frac{u\circ F_p-u^\star(p)-\widetilde L_z}{r}\right|d\mathscr L^n=C\int_{\widehat B(0,2)}\left|\frac{u\circ F_p\circ \delta_r- u^\star(p)}{r}-\widetilde L_z\right|d\mathscr L^n,
	\end{aligned}
	\]
which allows to conclude.
\end{proof}

The proofs of the following two results are postponed to Appendix \ref{app:risultatitecnici}.

\begin{proposition}[Properties of approximate differentiability points]\label{prop:DuBorel}
	Let $(\R^n,X)$ be an equiregular CC space, $\Omega$ be an open set in $\R^n$ and let $u\in L^1_{loc}(\Omega;\R^k)$.
	Then $\mathcal D_u$ is a Borel set and the map $D_X^{ap}u:~\mathcal D_u\rightarrow\R^{m\times k}$ is a Borel map.
\end{proposition}

\begin{proposition}[Locality]\label{prop:localita(balneare?)}
	Let $(\R^n,X)$ be an equiregular CC space, $\Omega$ an open set in $\R^n$ and $u,v\in L^1_{loc}(\Omega; \R^k)$. Suppose that $p\in \Omega$ is of density 1 for the set $\left\{q\in\Omega:u(q)=v(q)\right\}$. Then the following facts hold.
	\begin{itemize}
	\item[(i)] If $p\in \Omega\setminus\left(\mathcal S_u\cup\mathcal S_v\right)$, then $ u^\star(p)=v^\star(p)$. 
	\item[(ii)] If $p\in \mathcal J_u\cap\mathcal J_v$, then $(u^+(p),u^-(p),\nu_u(p))\equiv(v^+(p),v^-(p),\nu_v(p))$.
	\item[(iii)] If $p\in \mathcal D_u\cap\mathcal D_v$ then $D_X^{ap}u(p)=D_X^{ap}v(p)$.
	\end{itemize}
\end{proposition}

\subsection{Functions with bounded \texorpdfstring{$X$}{X}-variation}\label{sec:BVX}
In this section we review the definition and basic properties of $BV_X$ functions. We keep on working in a fixed equiregular CC space $(\R^n,X)$, while $\Omega$ denotes a fixed  open subset of $\R^n$.

\begin{definition}[Functions with bounded $X$-variation]
	We say that $u \in L^1_{loc}(\Omega)$ is a function of {\em locally bounded $X$-va\-ria\-tion} in $\Omega$, and we write $u \in BV_{X,loc}(\Omega)$, if there exists a $\R^m$-valued Radon measure $D_Xu=(D_{X_1}u,\dots,D_{X_m}u)$ in $\Omega$ such that for every open set $A\Subset \Omega$ and for every $\varphi\in C_c^1(A)$ we have
	\begin{equation}
	\forall\ i=1,\dots, m\qquad \int_A \varphi\: d(D_{X_i}u)=-\int_A u X_i^*\varphi \: d\mathscr{L}^n,\label{1}
	\end{equation}
	where $X_i^*$ denotes the formal adjoint of $X_i$. If $u\in L^1(\Omega)$,  we say that $u$ has {\em bounded $X$-variation} in $\Omega$, and we write $u \in BV_X(\Omega)$, if, moreover, the total variation $|D_Xu|$ of $D_Xu$ is finite on $\Omega$.
\end{definition}

As customary, we write $BV_X(\Omega;\R^k)\coloneqq(BV_X(\Omega))^k$, and similarly for $BV_{X,loc}(\Omega;\R^k)$. It can be useful to observe that if $u\in BV_X(\Omega;\R^k)$, the following inequalities hold
\begin{equation}\label{vectorbv}
\max_{1\leq i\leq k} |D_Xu^i|(\Omega)\leq |D_Xu|(\Omega)\leq \sum_{i=1}^k|D_Xu^i|(\Omega).
\end{equation}
	
The following approximation result is proved in \cite{FSSC3,GaroNhiCPAM}.
	
\begin{theorem}\label{teo:approxBV}
  Let $u\in BV_X(\Omega;\R^k)$. Then there exists a sequence $(u_h)$ in $C^\infty(\Omega; \R^k)$ such that
  \[
  \lim_h\|u_h-u\|_{L^1(\Omega;\R^k)}=0 \qquad \text{and}\qquad \lim_h|D_Xu_h|(\Omega)=|D_Xu|(\Omega).
  \]
\end{theorem}

We now state and prove a simple but useful result.	
	
\begin{proposition}\label{BVdiffeo}
	Let $\Omega,\widetilde\Omega$ be two open sets in $\R^n$ and let $G:\Omega\rightarrow \widetilde\Omega$ be a diffeomorphism. Let also $X_1,\dots, X_m$ be vector fields on $\Omega$ and define for every $i=1,\dots,m$ the vector fields $Y_i\coloneqq dG(X_i)$ on $\widetilde\Omega$. Then 
	\begin{equation}\label{eq:tagliaerba}
	u\in BV_{X,loc}(\Omega)\quad\Longleftrightarrow\quad v\coloneqq u\circ G^{-1} \in BV_{Y,loc}(\widetilde\Omega).
	\end{equation}
	More precisely, for every open set $U\Subset\Omega$ and setting $V\coloneqq G(U)$, one has for every $u\in BV_{X,loc}(\Omega)$ that
	\begin{equation}\label{eq:tagliaerba2}
	m|D_Xu|(U)\leq |D_Yv|(V)\leq M|D_Xu|(U)
	\end{equation}
	for $m\coloneqq \inf_U |\det \nabla G|$ and $M\coloneqq \sup_U |\det\nabla G|$.
\end{proposition}
\begin{proof}
We claim that, for any open set $U\Subset\Omega$ and any $u\in BV_{X,loc}(\Omega)$, one has
\[
v\coloneqq u\circ G^{-1} \in BV_{Y}(V)\qquad\text{and}\qquad |D_Yv|(V)\leq M|D_Xu|(U).
\]
This would be enough to conclude: indeed, the claim would imply both the $\Rightarrow$ implication in \eqref{eq:tagliaerba} and the second inequality in \eqref{eq:tagliaerba2}, while the $\Leftarrow$ implication in \eqref{eq:tagliaerba} and the first inequality in \eqref{eq:tagliaerba2} simply follow by replacing $X,U,u,G$ with (respectively) $Y,V,v,G^{-1}$ and noticing that $m=(\sup_V |\det\nabla (G^{-1})|)^{-1}$.

Let us prove the claim. First we assume that $u\in C^\infty(U)$, so that also $v$ is smooth on $V$. For every $\varphi\in C_c^1(V;\R^m)$ with $|\varphi|\leq 1$, by a change of variable we have that
\[
\int_V \langle Yv,\varphi\rangle d\mathscr L^n = \int_U\langle Xu,(\varphi\circ G) \rangle |\det\nabla G|d\mathscr L^n,
\]	
which gives 
	\[
	 |D_Yv|(V)\leq M|D_Xu|(U).
	\] 	
	In  case $u\in BV_X(U)$ is not smooth, consider a sequence $(u_h)$ in $C^\infty(U)$ that converges to $u$ in $L^1(U)$ and such that
	\[
	\lim_h |D_Xu_h|(U)=|D_X u|(U).
	\]
	Defining $v_h\coloneqq u_h\circ G^{-1}$, we easily get that $v_h$ converges to $v$ in $L^1(V)$ as $h\to +\infty$. Therefore 
	\[
	|D_Yv|(V)\leq \liminf_h |D_Yv_h| (V)\leq M\liminf_h|D_Xu_h|(U)=M|D_Xu|(U)
	\]
	and the proof is accomplished.
\end{proof}	

\begin{definition}[Sets with finite $X$-perimeter]
A measurable set $E\subset\R^n$ has {\em locally finite $X$-perimeter} (resp., {\em finite $X$-perimeter}) in $\Omega$ if $\chi_E\in BV_{X,loc}(\Omega)$ (resp., $\chi_E\in BV_{X}(\Omega)$). In such a case we define the {\em $X$-perimeter measure} $P_X^E$ of $E$ by $P_X^E\coloneqq |D_X\chi_E|$. 
	\end{definition}
	
It will sometimes be useful to write $P_X(E,\cdot)$ instead of $P_X^E$.
	
\begin{definition}[Measure theoretic horizontal normal]
  If  $E$ is a set with locally finite $X$-perimeter, then by Riesz representation theorem there exists a $P_X^E$-measurable function $\nu_E:\R^n\to \mathbb S^{m-1}$  such that
  \[
  D_X\chi_E=\nu_EP_X^E.
  \]
 We  call $\nu_E$ the {\em measure theoretic horizontal normal} to $E$. 
\end{definition}
  
The following result is proved in \cite{Ambrosio01} and it will be of capital importance in the following.

\begin{theorem}\label{teo:ambrosio}
Let $(\R^n,X)$ be an equiregular CC space of homogeneous dimension $Q$; let $E\subset\R^n$ be a set with finite $X$-perimeter in an open set $\Omega\subset\R^n$. Then 
\begin{equation}
P_X^E\res\Omega=\theta \mathscr H^{Q-1}\res (\Omega\cap\partial^\ast  E)
\end{equation}
for a suitable  positive function $\theta$ that is locally bounded away from 0. Moreover 
\[
\limsup_{r\to0}\frac{P_X^E(B(p,2r))}{P_X^E(B(p,r))}<\infty\qquad\text{for $P_X^E$-a.e. }p\in \Omega\cap\partial^\ast  E.
\]
\end{theorem}
  
The proofs of the following well-known result can be found, for instance, in \cite{FSSC3}.

\begin{theorem}[Coarea Formula for $BV_X$ functions]\label{coarea}
	Let $(\R^n,X)$ be a $CC$ space, let $\Omega$ be an open set in $\R^n$ and let $u\in BV_X(\Omega)$. Then, if we define $E_s\coloneqq \{p\in\Omega:u(p)>s\}$, we have
	\[
	|D_Xu|(\Omega)=\int_{-\infty}^{+\infty}P_X(E_s;\Omega)ds.
	\]
\end{theorem}

The next result is essentially  \cite[Theorem 1.2]{CapDanGar94}; note, however, that the dimension $Q$ appearing in  \cite[Theorem 1.2]{CapDanGar94} is slightly different from the homogeneous dimension we are considering. See also \cite{Jerison}.

\begin{theorem}\label{teo:poincare}
Let $\Omega$ be an open subset of an equiregular CC space $(\R^n,X)$ of homogeneous dimension $Q$ and let $K\subseteq\Omega$ be compact. Then there exist $C>0$ and $R>0$ such that, for every $p\in K$, $r\in(0,R)$ and $u\in BV_{X,loc}(\Omega;\R^k)$, the inequality
\[
\left(\fint_{B(p,r)} |u-u_{p,r}|^{\frac{Q}{Q-1}}\,d\mathscr L^n\right)^{\frac{Q-1}{Q}} \leq  \frac{C}{r^{Q-1}}|D_Xu|(B(p,r)),
\]
where $u_{p,r}\coloneqq \fint_{B(p,r)} u\,d\mathscr L^n$, holds.
\end{theorem}
\begin{proof}
  It is clearly enough to consider the case $k=1$. The proof then easily follows by \cite[Theorem 5.1]{HajKos} on taking into account Theorem \ref{ccproperties}, \cite[Theorem 1.1]{CapDanGar94}, \cite[Corollary 9.8 and Theorem 10.3] {HajKos} and Theorem \ref{teo:approxBV}.
\end{proof}

An easy consequence of Theorem \ref{teo:poincare} is the following isoperimetric inequality.

\begin{theorem}[Isoperimetric inequality in CC spaces]\label{isoperimetric}
	Let $(\R^n,X)$ be an equiregular CC space and let $K\subset\R^n$ be a compact set. Then there exist $C>0$ and $R>0$ such that, for every $p\in K$, $r\in(0,R)$ and every ${\mathscr L}^n$-measurable set $E\subseteq \R^n$, one has
	\[
	\min\left\{\mathscr{L}^n(E\cap B(p,r)),\mathscr{L}^n(B(p,r)\setminus E)\right\}^{\frac{Q-1}{Q}}\leq C P_X(E,B(p,r)).
	\]
\end{theorem}

We conclude this section with some auxiliary results. The first one is proved in \cite{DVimmcpt}.

\begin{theorem}\label{teo:applicazione}
Let $X=(X_1,\dots,X_m)$ and $X^j=(X_1^j,\dots,X_m^j)$, $j\in \mathbb N$, be $m$-tuples of linearly independent smooth vector fields on $\R^n$ such that $X$ satisfies the H\"ormander condition and its CC balls are bounded in $\R^n$; assume that, for every $i=1,\dots,m$, $X^j_i\to X_i$ in $C^\infty_{loc}(\R^n)$ as $j\to\infty$.  Let $u_j\in BV_{X^j,loc}(\R^n)$ be a sequence of functions that is locally uniformly  bounded in $BV_{X^j}$, i.e., such that for any compact set $K\subset\R^n$ there exists $M>0$ such that
\[
\forall j\in\mathbb N\qquad \|u_j\|_{L^1(K)} + |D_{X^j}u_j|(K)\leq M<\infty.
\]
Then, there exist $u\in BV_{X,loc}(\R^n)$ and a subsequence $(u_{j_h})$ of $(u_j)$ such that $u_{j_h}\to u$ in $L^1_{loc}(\R^n)$ as $h\to\infty$. Moreover, for any bounded open set $\Omega\subset\R^n$ one has
\[
|D_Xu|(\Omega) \leq \liminf_{j\to\infty} |D_{X^j} u_j|(\Omega).
\]
\end{theorem}

The proof of Theorem \ref{teo:applicazione}  given in \cite{DVimmcpt}  implicitly contains also the following result's proof, that we however provide for the sake of completeness.

\begin{proposition}\label{prop:semicontstrutturavariabile}
Let $X=(X_1,\dots,X_m)$ and $X^j=(X_1^j,\dots,X_m^j)$, $j\in \mathbb{N}$, be $m$-tuples of linearly independent smooth vector fields on $\mathbb R^n$ such that, for every $i=1,\dots,m$, $X^j_i\to X_i$ in $C^\infty_{loc}(\R^n)$ as $j\to\infty$.  Let $(u_j)\subset L^1_{loc}(\R^n)$ be a sequence converging in $L^1_{loc}(\R^n)$ to some $u$; then, for any open bounded set $\Omega\subset\R^n$ one has
\[
|D_{X} u|(\Omega) \leq\liminf_{j\to\infty} |D_{X^j} u_j|(\Omega)
\]
\end{proposition}
\begin{proof}
For any $i=1,\dots,m$ and any $j\in \mathbb N$ we  write
\[
X_i(x)=\sum_{k=1}^n a_{i,k}(x)\partial_k\qquad\text{and}\qquad X_i^j(x)=\sum_{k=1}^n a_{i,k}^j(x)\partial_k
\]
for suitable smooth functions $a_{i,k},a_{i,k}^j$. Then, for any test function $\varphi\in C^1_c(\Omega;\R^m)$ we have
\begin{equation}\label{eq:convdebole}
\begin{split}
 \int_\Omega u \sum_{i=1}^m X_i^\ast \varphi_i\:d\mathscr L^n =&  \int_\Omega u\sum_{i=1}^m\sum_{k=1}^n\partial_k(a_{i,k}\varphi_i)\:d\mathscr L^n
\ = \ \lim_{j\to\infty} \int_\Omega u_j\sum_{i=1}^m\sum_{k=1}^n\partial_k(a_{i,k}^j\varphi_i)\:d\mathscr L^n\\
 = & \lim_{j\to\infty} \int_\Omega u_j \sum_{i=1}^m {X_i^j}^\ast \varphi_i\:d\mathscr L^n
\ \leq \ \|\varphi\|_{L^\infty(\Omega)} \liminf_{j\to\infty} |D_{X^j} u_j|(\Omega).
\end{split}
\end{equation}
The proof is accomplished.
\end{proof}

\begin{rmk}\label{rem:convergenzadebole}
Let $X,X^j,u_j,u$ be as in Proposition \ref{prop:semicontstrutturavariabile} and assume that $|D_{X^j}u^j|$ are locally uniformly bounded in $\R^n$, i.e., for any compact set $K\subset\R^n$ there exists $C_K<\infty$ such that $|D_{X^j}u^j|(K)<C_K$ for all $j$. Then $D_{X^j}u^j$ weakly$^\ast$ converges to $D_X u$  in $\R^n$. 

Indeed, one can reason as in \eqref{eq:convdebole} to show that for any test function $\varphi\in C^1_c(\R^n)$ and any $i=1,\dots,m$
\[
\lim_{j\to\infty}\int\varphi\: dD_{X^j_i}u^j = \int\varphi\: dD_{X_i}u
\]
and  the density of $C^1_c$ in $C^0_c$ allows to conclude.
\end{rmk}

\section{Fine properties of BV functions}\label{sec:dimostrazioni}
This section is devoted to the proof of our main results.

\begin{lemma}\label{th:3.74}
	Let $(\R^n,X)$ be an equiregular CC space, let $\Omega\subset\R^n$ be  open and let $(E_h)$ be a sequence of measurable sets in $\Omega$ such that
	\[
	\text{$\displaystyle\lim_h\mathscr L^n(E_h)=0$\quad and\quad $\displaystyle\lim_h P_X(E_h;\Omega)=0$.}
	\]
	Then for every $\alpha\in(0,1)$ we have
	\[
	\mathscr H^{Q-1}\left(\bigcap_{h=1}^\infty\left\{p\in\Omega:\limsup_{r\to 0}\frac{\mathscr L^n(E_h\cap B(p,r))}{\mathscr L^n(B(p,r))}\geq\alpha\right\}\right)=0.
	\]
\end{lemma}
\begin{proof}
Set
	\[
	E_h^\alpha\coloneqq \left\{q\in\Omega:\limsup_{r\to 0}\frac{\mathscr L^n(E_h\cap B(q,r))}{\mathscr L^n(B(q,r))}\geq\alpha\right\},
	\]
and suppose without loss of generality that $\mathscr L^n(E_h)>0$ for every $h\in \mathbb N$. Let $K\Subset\Omega$. By Theorem \ref{ccproperties} there exist $C>1$ and $R>0$ such that for every $q\in K$, for every $0<r<2R$ we have
\[
\frac{1}{C}r^Q\leq\mathscr L^n(B(q,r))\leq Cr^Q.
\]
For any sufficiently large $h\in \mathbb N$ we have 
	\[
	\left(\frac{2C\mathscr L^n(E_{h})}{\alpha}\right)<R^Q.
	\]
	Fix now $p\in E_h^\alpha \cap K$ and define $\delta_h = \left(\frac{4C\mathscr L^n(E_h)}{\alpha}\right)^{1/Q}$; then
	\[
	\frac{\mathscr L^n(E_h\cap B(p,\delta_h))}{\mathscr L^n(B(p,\delta_h))}\leq \frac{C\mathscr L^n(E_h)}{\delta_h^Q}=\frac{\alpha}{4}.
	\]
	On the other hand, by definition of $E_h^\alpha$ we can find arbitrarily small radii $r>0$ such that
	\[
	\frac{\mathscr L^n(E_h\cap B(p,r))}{\mathscr L^n(B(p,r))}\geq \frac{\alpha}{2}.
	\]
	Taking  into account Proposition \ref{stella}, a continuity argument allows us to find $0<\varrho\leq\delta_h$ such that
\[
\mathscr L^n(E_h\cap B(x,\varrho))=\frac{\alpha}{2}\mathscr L^n(B(x,\varrho)).
\]
By the $5r$-covering Lemma, we can find a family $\displaystyle\left\{B(p_j,\varrho_j):j\in \mathbb N\right\}$ of pairwise disjoint balls in $\Omega$  such that, for every $j\in\mathbb N$,  
	\begin{align}
	& p_j\in E_h^\alpha\cap K\nonumber\\
	& \mathscr L^n(E_h\cap B(p_j,\varrho_j))=\frac{\alpha}{2}\mathscr L^n(B(p_j,\varrho_j))\nonumber\\
	& E_h^\alpha\cap K\subseteq \bigcup_{j=0}^\infty B(p_j,5\varrho_j).\label{eq:boh}
\end{align}	
Since $\mathscr L^n(E_h)$ is finite, by Theorem \ref{isoperimetric} we get $M>0$ such that 
	\begin{align*}
	\frac{\alpha}{2C}\varrho^Q_j\leq \frac{\alpha}{2}\mathscr L^n(B(p_j,\varrho_j))=\mathscr L^n(E_h\cap B(p_j,\varrho_j))
	\leq \big(M\:P_X(E_h;B(p_j,\varrho_j))\big)^{\frac{Q}{Q-1}}.
	\end{align*}
	Therefore we have that for every $j\in \mathbb N$
	\[
	\varrho_j^{Q-1}\leq M\left(\frac{2C}{\alpha}\right)^{\frac{Q-1}{Q}}P_X(E_h;B(p_j,\varrho_j)).
	\]
	Finally
	\begin{align*}
	\mathscr H^{Q-1}_{10\delta_h}\Big(K\cap \bigcap_{i=0}^\infty E_i^\alpha\Big) &\leq\mathscr H^{Q-1}_{10\delta_h}(K\cap E_h^\alpha) 
	\stackrel{\eqref{eq:boh}}{\leq} \omega _{Q-1}5^{Q-1}\sum_{j=0}^\infty \varrho_j^{Q-1}\\
	& \leq \omega_{Q-1}5^{Q-1}M \left(\frac{2C}{\alpha}\right)^{\frac{Q-1}{Q}}\sum_{j=0}^\infty P_X(E_h;B(p_j,\varrho_j))\\
	& \leq \omega_{Q-1}5^{Q-1}M \left(\frac{2C}{\alpha}\right)^{\frac{Q-1}{Q}} P_X(E_h;\Omega).
	\end{align*}
	Taking the limit for $h\to \infty$ we get
	\[
	\mathscr H^{Q-1}\left(K\cap \bigcap_{i=0}^\infty E_i^\alpha\right)=0.
	\]
	By the arbitrariness of $K$, the proof is complete.
\end{proof}

Before passing to the next result, we introduce some notation that we are going to use frequently in what follows. Let $p\in\R^n$ be fixed and let $F_p$ denote adapted  exponential coordinates as in  \eqref{eq:defcoord}, for a fixed choice of a basis $Y_1,\dots,Y_n$ as in \eqref{eq:defcoord}. Given $r>0$ and $i\in\{1,\dots,m\}$, define 
\begin{equation}\label{eq:defcampiriscalati}
\widetilde X_i^r\coloneqq r (d\delta_{r^{-1}}) [\widetilde X_i\circ\delta_{r}].
\end{equation}
If $\widetilde d_r,\widetilde B_r(x,\varrho)$ denote, respectively,  distance and balls with respect to the metric induced by the vector fields $(\widetilde X_1^r,\dots,\widetilde X_m^r)$, it is easy to see that the dilations $\delta_{r}$ satisfy
	\[
	\widetilde d_r(\xi,\eta) =\frac 1r\widetilde d(\delta_r\xi,\delta_r\eta).
	\]
By Theorem \ref{th:localgroup}, the convergence
\begin{equation}\label{eq:convGH}
\lim_{r\to0} \widetilde B_r(0,\varrho) = \widehat B(0,\varrho)
\end{equation}
holds in the Gromov-Hausdorff sense, $\widehat B(0,\varrho)$ denoting a ball in the tangent Carnot group at $p$. Moreover, given $u\in BV_{X,loc}(\R^n;\R^k)$  we set
\begin{equation}\label{eq:deftildeu}
\widetilde u\coloneqq u\circ F_p\qquad\text{and}\qquad\widetilde u_r\coloneqq \widetilde u\circ \delta_r;
\end{equation}
notice that $|D_{\widetilde{X}^r} \widetilde u_r|(\widetilde B_r(0,\varrho))=r^{1-Q}|D_{\widetilde X} \widetilde u|(\widetilde B(0,r\varrho))$.

We can now prove the following lemma.

\begin{lemma}\label{negligibleset}
	Let $(\R^n,X)$ be an equiregular CC space, let $\Omega\subset\R^n$ be  open and consider $u\in BV_X(\Omega;\R^k)$. Then
	\[
	\mathscr H^{Q-1}\left(\left\{p\in\Omega:\limsup_{r\to 0}\fint_{B(p,r)}|u|^{\tfrac{Q}{Q-1}}\:d\mathscr L^n=+\infty\right\}\right)=0.
	\]
\end{lemma}
\begin{proof}
	We can suppose without loss of generality that $k=1$. Possibly considering $|u|$ instead of $u$, we can suppose that $u\geq 0$; we also assume without loss of generality that $\Omega$ is bounded in $\R^n$. Define the set
	\[
	D=\left\{p\in\Omega :\limsup_{r\to 0}\frac{|D_Xu|(B(p,r))}{r^{Q-1}}=+\infty\right\}.
	\]
	By Proposition \ref{k-density} we have that $\displaystyle\mathscr H^{Q-1}(D)=0$. For every $h\in \mathbb N$ we can find $t_h\in (h,h+1)$ such that
	\[
	P_X(\{u>t_h\},\Omega)\leq \int_h^{h+1}P_X(\{u>t\},\Omega)dt.
	\]
	Define $E_h=\{u>t_h\}$. Since $u\in L^1(\Omega)$ we have that $\lim_h\mathscr L^n(E_h)=0$ and applying the Coarea Formula of Theorem \ref{coarea} we get
	\[
	\sum_{h=0}^\infty P_X(E_h,\Omega)\leq\int_0^{+\infty} P_X(\{u>t\},\Omega)dt=|D_X u|(\Omega)<+\infty,
	\]
and therefore $\lim_hP_X(E_h,\Omega)=0$. We are in a position to apply Lemma \ref{th:3.74}. Defining for every $h\in \mathbb N$
	\[
	F_{h}=\left\{p\in\Omega:\limsup_{r\to0}\frac{\mathscr L^n(E_h\cap B(p,r))}{\mathscr L^n(B(p,r))}\geq \alpha\right\},
	\]
where $\alpha>0$ will be chosen later depending on $\Omega$ only, we have that $\mathscr H^{Q-1}\left(\bigcap_{h=0}^\infty F_{h}\right)=0$.	It is then sufficient to prove the inclusion 
	\[
		L\coloneqq\left\{p\in\Omega: \limsup_{r\to 0} \fint_{B(p,r)}|u|^{\tfrac{Q}{Q-1}}d\mathscr L^n=+\infty\right\}\subseteq D\cup\bigcap_{h=0}^\infty F_{h}.
	\]
	To this aim, we fix $p\notin D\cup\bigcap_{h=0}^\infty F_{h}$ and we prove that $p\notin L$. Define $	u_{p,r}\coloneqq\fint_{B(p,r)}ud\mathscr L^n$. Applying  Theorem \ref{teo:poincare} we get $C>0$ and $R>0$ such that for every $q\in \Omega$ and all $0<r<R$
	\begin{equation}\label{eq:poincare}
	\fint_{B(q,r)}|u(y)-u_{q,r}|^{\tfrac{Q}{Q-1}}d\mathscr L^n(y)\leq C\left(\frac{|D_X u|(B(q,r))}{r^{Q-1}}\right)^{\tfrac{Q}{Q-1}}.
	\end{equation}
	It is enough to prove that $\limsup_{r\to 0}u_{p,r}<+\infty$: in this case, in fact,   inequality \eqref{eq:poincare} and  the definition of $D$ would imply that $p \notin L$. 
	
	By contradiction we find an infinitesimal sequence $(r_j)$ such that  $\lim_ju_{p,r_j}=+\infty$. Define 	$\widetilde u,\widetilde u_{r_j}$ as in \eqref{eq:deftildeu} (with $r=r_j$) and $\widetilde v_j\coloneqq \widetilde u_{r_j}-u_{p,r_j}$; set also
	\[
	\text{$\widetilde X_i^j\coloneqq \widetilde X_i^{r_j}$\qquad and\qquad $\widetilde X^j\coloneqq (\widetilde X_1^j, \dots,\widetilde X_m^j)$.}
	\]
Since $p\notin D$, for any $\varrho>0$ the sequence $r_j^{1-Q}|D_{X}u|(B(p,\varrho r_j))$ is uniformly bounded with respect to $j\in \mathbb N$;  by Proposition \ref{BVdiffeo}, the same is true for the sequence 
	\[
	|D_{\widetilde X^j}\widetilde v_j|(\widetilde B_j(0,\varrho)) = r_j^{1-Q}|D_{\widetilde X}\widetilde u|(\widetilde B(0,\varrho r_j)),
	\]
	where $\widetilde B_j(0,\varrho)\coloneqq \widetilde B_{r_j}(0,\varrho)$ according to the notation introduced after \eqref{eq:defcampiriscalati}.  Taking also \eqref{eq:convGH} into account, this proves that, for any compact set $K\subseteq\R^n$, the sequence $|D_{\widetilde X^j}\widetilde v_j|(K)$ is bounded; by \eqref{eq:poincare}, also $\|\widetilde v_j\|_{L^1(K)}$ is bounded.
	
		By Theorem \ref{teo:applicazione} (recalling also Theorem \ref{th:tangentspace}) there exists $w\in L^1(\widehat B(0,1))$ such that, possibly extracting a subsequence, $\widetilde v_j\to w$ in $L^1(\widehat B(0,1))$. 
	Consequently, for almost every $x\in \widehat B(0,1)$ we have
	\[
	\lim_j u(F_p(\delta_{r_j}x))=+\infty
	\]
	and then, for every $h\in \mathbb N$,
	\[
	\begin{aligned}
		\mathscr L^n(\widehat B(0,1))&=\lim_j\mathscr L^n(\{x\in \widetilde B_j(0,1): u(F_p(\delta_{r_j}x))>t_h\})\\
		&=\lim_j r_j^{-Q}\mathscr L^n(\{x\in \widetilde B(0,r_j): u(F_p(x))>t_h\})\\
		&=\lim_j\frac 1{r_j^Q} \int_{B(p,r_j)\cap E_h}|\det\nabla F_p^{-1}|\:d\mathscr L^n\\
		&\leq |\det\nabla F_p^{-1}(p)|\limsup_{r\to0} \frac{\mathscr L^n(E_h\cap B(p,r))}{\mathscr L^n(B(p,r))} \: \frac{\mathscr L^n(B(p,r))}{r^Q}\\
		&\leq \frac C{|\det\nabla F_p(0)|}\limsup_{r\to0} \frac{\mathscr L^n(E_h\cap B(p,r))}{\mathscr L^n(B(p,r))} 
	\end{aligned}
	\]
where $C>0$ is given by Theorem \ref{ccproperties} with $K=\overline\Omega$. Notice that $\mathscr L^n(\widehat B(0,1))$ depends on $p$. Using \eqref{eq:convGH} we obtain
\begin{align*}
&\limsup_{r\to0} \frac{\mathscr L^n(E_h\cap B(p,r))}{\mathscr L^n(B(p,r))} \\
\geq & \frac {|\det\nabla F_p(0)|}C\mathscr L^n(\widehat B(0,1))
= \frac {|\det\nabla F_p(0)|}C\lim_{r\to 0} \mathscr L^n(\widetilde B_r(0,1))\\
=& \frac {|\det\nabla F_p(0)|}C \lim_{r\to 0} \frac{1}{r^Q}\mathscr L^n( \widetilde B(0,r))= \frac {|\det\nabla F_p(0)|}C \lim_{r\to 0}  \frac{1}{r^Q} \int_{B(p,r)}|\det \nabla F_p^{-1}|\:d\mathscr L^n\\
\geq & \frac {1}C \liminf_{r\to 0}  \frac{\mathscr L^n(B(p,r))}{r^Q} \ \geq\ \frac 1{C^2}.
\end{align*}
This proves that  $p\in \bigcap_{h=0}^\infty F_{h}$ for $\alpha\coloneqq 1/C^2$, a contradiction.
\end{proof}

The following proposition contains some of the first ``fine'' properties of $BV_X$ functions we are interested in.

\begin{proposition}\label{prop3.76}
	Let $(\R^n,X)$ be an equiregular $CC$ space. Then there exists $\lambda:\R^n\to(0,+\infty)$ locally bounded away from 0 such that, for every open set $\Omega\subset\R^n$  and every $u\in BV_X(\Omega;\R^k)$
	\[
	|D_Xu|\geq \lambda|u^+-u^-| \mathscr S^{Q-1}\res \mathcal J_u
	\]
	 and for every Borel set $B\subseteq \Omega$ the following implications hold:
	\begin{align}
		&\mathscr H^{Q-1}(B)=0\quad \Rightarrow\quad |D_Xu|(B)=0;\label{1implication}\\
&	\label{2implication}
		\mathscr H^{Q-1}(B)<+\infty \text{ and } B\cap\mathcal{S}_u=\emptyset\quad \Rightarrow \quad |D_Xu|(B)=0.
	\end{align}
\end{proposition}
\begin{proof}
Let us prove the first part of the statement; we assume without loss of generality that $k=1$. Consider $p\in \mathcal J_u$. By Proposition \ref{jumpequiv} the sequence $\widetilde u_r\coloneqq u\circ F_p\circ\delta_{r}$ converges in $L^1(\widehat B(0,1))$ as $r\to 0$ to the function
	\[
	w_p(y)\coloneqq \begin{cases}
	u^+(p) & \text{if } \widetilde L_{\nu(p)}(y) \geq 0\\
	u^-(p) & \text{if } \widetilde L_{\nu(p)}(y) <0.
	\end{cases}
	\]
	Defining $	\widetilde{X}^r_i$ as in \eqref{eq:defcampiriscalati} and using Propositions \ref{prop:semicontstrutturavariabile} and \ref{BVdiffeo} we obtain for any positive $\varepsilon$ that
\[
\begin{split}
 \liminf_{r\to0}\frac{|D_Xu|(B(p,r))}{r^{Q-1}}\geq & |\det \nabla F_p(0)|\ \liminf_r |D_{\widetilde X^r}\widetilde u_{r}|(\widetilde B_r(0,1))\\
	  \geq & |\det \nabla F_p(0)|\ \liminf_r |D_{\widetilde X^r}\widetilde u_{r}|(\widehat B(0,1-\varepsilon))\\
	   \geq &|\det \nabla F_p(0)|\ |D_{\widehat{X}} w_p|(\widehat B(0,1-\varepsilon)),
		\end{split}
\]
	whence
\begin{equation}\label{inequality3.76}
	\begin{split}
	\liminf_{r\to0}\frac{|D_Xu|(B(p,r))}{r^{Q-1}}\geq & |\det \nabla F_p(0)|\ |D_{\widehat{X}} w_p|(\widehat B(0,1))\\
	\geq & |\det \nabla F_p(0)| |u^+(p)-u^-(p)|\mathscr H^{n-1}_{e}(\overline \nu^\perp\cap\widehat B(0,1))
\end{split}
	\end{equation}
where $\overline \nu\coloneqq (\nu_1,\dots,\nu_m,0,\dots,0)\in\R^n$ and $\mathscr H^{n-1}_{e}$ denotes the Euclidean Hausdorff measure in $\R^n$. It is easily seen that, for any $p\in\R^n$, there exist $c>0$ and a neighborhood $U$ of $p$ such that the function $\lambda(q)\coloneqq |\det \nabla F_q(0)| \mathscr H^{n-1}_{e}(\overline \nu^\perp\cap\widehat B_q(0,1))$ is such that $\lambda\geq c$ on $U$. By Corollary \ref{k-densityf}, this proves the first part of the statement.
	
By Theorem \ref{teo:ambrosio}, the implication  \eqref{1implication} is  true in case $k=1$ and $u=\chi_E$ for some $E\subset \R^n$ with  finite $X$-perimeter.	If $k=1$ and $u\in BV_X(\Omega)$, we define $E_s\coloneqq \{u>s\}$ and we apply Theorem \ref{coarea} (and, again, Theorem \ref{teo:ambrosio}) to get
	\[
	|D_Xu|(B)=\int_{-\infty}^{+\infty}P_X(E_s;B)ds=\int_{-\infty}^{+\infty}\left(\int_{ B\cap \partial^* E_s}\theta_s d\mathscr H^{Q-1}\right)ds
	\]
for  suitable positive functions $\theta_s$. This allows to infer \eqref{1implication}. In the general case $k\geq 1$, it is sufficient to recall inequality \eqref{vectorbv}.

	In order to prove \eqref{2implication} we consider $u\in BV_X(\Omega;\R^k)$  and  a Borel subset $B$ of $\Omega$  such that $B\cap \mathcal S_u=\emptyset$. If $k=1$, by Theorem \ref{coarea} we obtain again
	\[
	\begin{aligned}
	|D_Xu|(B)&=\int_{-\infty}^{+\infty}\left(\int_{B\cap \partial^*E_s}\theta_s \:d \mathscr H^{Q-1}\right)ds
=\int_{B}\int_\R\theta_s(p) \chi_{\partial^* E_s}(p)\:ds \:d\mathscr H^{Q-1}(p)=0,
	\end{aligned}
	\]
	the last equality following from Proposition \ref{levelsets}. In the case $u\in BV_X(\Omega;\R^k)$ with $k\geq2$, it is sufficient to notice that $B\cap \mathcal S_u=\emptyset$ implies $B\cap \mathcal S_{u^\alpha}=\emptyset$ for every $\alpha=1,\dots, k$, and one concludes using inequality \eqref{vectorbv}.
\end{proof}

We now prove some of our main results.

\begin{proof}[Proof of Theorems \ref{teo:Suquasirettificabile} and  \ref{federervolpert}]
It is not restrictive to suppose $k=1$. We first prove Theorem \ref{teo:Suquasirettificabile}.

By the Coarea Formula we get a countable and dense set $D\subseteq \R$ such that for every $t\in D$ the level set $\{u>t\}$ has finite $X$-perimeter. We prove that
\begin{equation}\label{eq:Suquasirettif}
\mathcal S_u\setminus L\subseteq \bigcup_{t\in D} \partial^*\{u>t\}
\end{equation}
where, as in Lemma \ref{negligibleset}, $L$ denotes the $\mathscr H^{Q-1}$-negligible set 
\[
\left\{p\in\Omega: \limsup_{r\to 0} \fint_{B(p,r)}|u|^{\tfrac{Q}{Q-1}}d\mathscr L^n=+\infty\right\}.
\]
Theorem \ref{teo:Suquasirettificabile} is immediately implied by 
formula \eqref{eq:Suquasirettif}. In order to prove the latter, take $p\notin L$ and suppose that $p\notin \bigcup_{t\in D} \partial^*\{u>t\}$; we will prove that $p\notin \mathcal S_u$. By definition, $p$ is either a point of density $1$ or a point of density $0$ in $\{u>t\}$ for every $t\in D$. Notice that for every $t\in D\cap (0,+\infty)$ one has
	\[
	\frac{\mathscr L^n\left(\{u>t\}\cap B(p,r)\right)}{\mathscr L^n(B(p,r))}\leq \frac{1}{t}\fint_{B(p,r)}|u|d\mathscr L^n\leq \frac{1}{t}\left(\fint_{B(p,r)}|u|^{\frac{Q}{Q-1}}d\mathscr L^n\right)^{\frac{Q-1}Q}
	\] 
and therefore, if $t\in D\cap (0,+\infty)$ is large enough, $p$ is a point of density $0$ for $\{u>t\}$. Analogously, if $t\in D\cap (-\infty,0)$ and $-t$ is large enough, $p$ is a point of density $1$ for $\{u>t\}$. Hence we can find a real number
	\[
	z=z(p)\coloneqq \sup\left\{t\in D: \{u>t\} \text{ has density } 1 \text{ at } p \right\}.
	\] 
	By the density of $D$ in $\R$ we get that, for every $t>z$, $\{u>t\}$ has density $0$ at $p$ and, for every $t<z$, $\{u>t\}$ has density $1$ at $p$.\\
	We prove now that $z$ is the approximate limit of $u$ at $p$. To this end define $E_\varepsilon\coloneqq\{|u-z|>\varepsilon\}$ and estimate
	\[
	\begin{aligned}
	\frac{1}{r^Q}\int_{B(p,r)} |u-z|d\mathscr L^n&\leq \varepsilon C +\frac{1}{r^Q}\int_{E_\varepsilon\cap B(p,r)} |u-z|d\mathscr L^n\\
	& \leq \varepsilon C +\frac{1}{r^Q}\left(\mathscr L^n(E_\varepsilon\cap B(p,r))\right)^{1/Q}\left(\int_{B(p,r)}|u-z|^{\frac{Q}{Q-1}}d\mathscr L^n\right)^{\frac{Q-1}{Q}}\\
	& = \varepsilon C +\left(\frac{\mathscr L^n(E_\varepsilon\cap B(p,r))}{r^Q}\right)^{1/Q}\left(\frac{1}{r^Q}\int_{B(p,r)}|u-z|^{\frac{Q}{Q-1}}d\mathscr L^n\right)^{\frac{Q-1}{Q}}.
	\end{aligned}
	\] 
	Since both $\{u>z+\varepsilon\}$ and $\{u<z-\varepsilon\}$ have density $0$ at $p$, one has
	\[
	\lim_{r\to 0} \frac{\mathscr L^n(E_\varepsilon\cap B(p,r))}{r^Q}=0
	\]
and, since $p\notin L$, we get
	\[
	\limsup_{r\to 0}\frac 1{r^Q}\int_{B(p,r)}|u-z|d\mathscr L^n\leq C\varepsilon,
	\]
from which we deduce that $p\notin \mathcal S_u$, as desired.

We now prove Theorem \ref{federervolpert}. When property $\mathcal R$ holds, the countable $X$-rectifiability of $\mathcal S_u$ immediately follows from \eqref{eq:Suquasirettif}. We have  to prove that $\mathscr H^{Q-1}(\mathcal S_u\setminus \mathcal J_u)=0$.  Let $\nu=\nu_{\mathcal S_u}$ be the horizontal normal to $\mathcal S_u$ and recall the notation $B^\pm_\nu(p,r)$  introduced in \eqref{eq:Bpmnu}.  By   Proposition \ref{rectifiabletraces} below, for $\mathscr H^{Q-1}$-almost every $p\in \mathcal S_u$ there exist $u^+(p)$ and $u^-(p)$ in $\R^k$ such that
	\[
	\lim_{r\to 0}\frac{1}{r^Q}\int_{B^+_{\nu(p)}(p,r)}|u-u^+(p)|d\mathscr L^n=\lim_{r\to 0}\frac{1}{r^Q}\int_{B^-_{\nu(p)}(p,r)}|u-u^-(p)|d\mathscr L^n=0.
	\]
Notice that $u^+(p)\neq u^-(p)$, for otherwise $u$ would have an approximate limit at $p$. This implies that $p$ is an approximate $X$-jump point associated with the triple $(u^+(p),u^-(p),\nu(p))$, and this concludes the proof.
\end{proof}

A milder version of Theorem \ref{federervolpert} holds when $(\R^n,X)$ satisfies the weaker property $\mathcal{LR}$, that we now introduce.

\begin{definition}[Property $\mathcal{LR}$]\label{def:proprietaLR}
	Let $(\R^n,X)$ be an equiregular CC space with homogeneous dimension $Q\in\mathbb N$. We say that $(\R^n,X)$ satisfies the {\em property $\mathcal{LR}$} if, for every open set $\Omega\subset\R^n$ and every $E\subseteq \R^n$ with locally finite $X$-perimeter in $\Omega$, the essential boundary $\partial^\ast E\cap\Omega$ is countably $X$-Lipschitz rectifiable. 
\end{definition}

The proof of the following result is an immediate consequence of \eqref{eq:Suquasirettif}.

\begin{theorem}\label{teo:SuLR}
	Let $(\R^n,X)$ be an equiregular CC space satisfying property $\mathcal{LR}$ and let $u\in BV_X(\Omega; \R^k)$. Then $\mathcal S_u$ is countably $X$-Lipschitz rectifiable.
\end{theorem}

Before proving Proposition \ref{rectifiabletraces}, that we used in the proof of Theorem \ref{federervolpert}, we state the following theorem, which is a consequence of some results contained in \cite{Vittone2}. We use the notation
\[
B_{f}^{\pm}(p,r)\coloneqq \{q\in B(p,r): \pm f(q) >0\}.
\]

\begin{theorem}\label{charactraces}
	Let $(\R^n,X)$ be an equiregular CC space, let $\Omega\subset\R^n$ be an  open set and let $f\in C^1_X(\Omega)$ be such that $Xf\neq 0$ on $\Omega$; let  $S$ be the $C^1_X$ hypersurface $S\coloneqq\{p\in\Omega:f(p)=0\}$. Then  there exist  linear operators $T^+,T^-:BV_{X,loc}(\Omega;\R^k)\rightarrow L^1_{loc}(S,\mathscr H^{Q-1})$ such that, for any $u\in BV_{X,loc}(\Omega;\R^k)$, one has for $\mathscr H^{Q-1}$-a.e. $p\in S$
	\[
	\lim_{r\to 0}\frac{1}{r^Q}\int_{B^+_f(p,r)}|u-T^+u(p)|d\mathscr{L}^n=\lim_{r\to 0}\frac{1}{r^Q}\int_{B^-_f(p,r)}|u-T^-u(p)|d\mathscr{L}^n=0.
	\]
In particular, for $\mathscr H^{Q-1}$-a.e. $p\in S$
	\[
	T^\pm u(p)=\lim_{r\to 0}\frac{1}{r^Q}\int_{B^\pm_f(p,r)} u\:d \mathscr L^n.
	\]
\end{theorem}

We can now prove the following proposition, where we implicitly use Remark \ref{jumpwellposed}.

\begin{proposition}\label{rectifiabletraces}
	Let $(\R^n,X)$ be an equiregular CC space and let $\Omega\subseteq \R^n$ be an open set. Let $R\subseteq \Omega$ be a countably $X$-rectifiable set with horizontal normal $\nu_R$. Then, for every $u\in BV_X(\Omega;\R^k)$ and for $\mathscr H^{Q-1}$-almost every $p\in R$ there exists a couple $(u^+_R(p),u^-_R(p))\in\R^k\times \R^k$ such that 
	\begin{equation}\label{eq:traccerett}
	\lim_{r\to 0}\frac{1}{r^Q}\int_{\Omega \cap B^+_{\nu_R(p)}(p,r)} |u-u^+_R(p)|d\mathscr L^n=\lim_{r\to 0}\frac{1}{r^Q}\int_{\Omega \cap B^-_{\nu_R(p)}(p,r)} |u-u^-_R(p)|d\mathscr L^n=0.
	\end{equation}
Moreover, if $(\R^n,X)$ satisfies property $\mathcal R$ and $R=\mathcal J_u$\footnote{The jump set $\mathcal J_u$ is countably $X$-rectifiable by Theorem \ref{federervolpert} and Remark \ref{rem:JusubsetSu}.}, then  $(u^+_{\mathcal J_u}(p),u^-_{\mathcal J_u}(p),\nu_{\mathcal J_u}(p))$ is an approximate $X$-jump triple for $u$ at $p$ in the sense of Definition \ref{approximatejump}.
\end{proposition}
\begin{proof}
	We can assume without loss of generality that $k=1$. Let $u\in BV_X(\Omega)$ be fixed. By definition of countable $X$-rectifiability we can find  a family $\{S_i:i\in \mathbb N\}$ of $C^1_X$ hypersurfaces in $\R^n$ such that
	\[
	\mathscr H^{Q-1}\left(R\setminus \bigcup_{i=0}^\infty S_i\right)=0.
	\]
	For every $i\in \mathbb N$ we can write, at least locally, $S_i=\{f_i=0\}$ and we can suppose that $Xf_i\neq 0$ on $S_i$. Formula \eqref{eq:traccerett} easily follows (with $u^\pm_R(p)=T^\pm u(p)$) from Theorem \ref{charactraces} for $\mathscr H^{Q-1}$-a.e. $p\in R$ such that $\#\{i\in\mathbb N:p\in S_i\}=1$. It is then enough to show that, for any fixed couple $i,j\in\mathbb N$ with $i\neq j$ and for $\mathscr H^{Q-1}$-almost every   $p\in S_i\cap S_j$, the equivalence
	\begin{equation}\label{eq:equivalenza}
	(T_i^+u(p),T^-_iu(p),\nu_{S_i}(p))\equiv (T_j^+u(p),T^-_ju(p),\nu_{S_j}(p))
	\end{equation}
holds. Here, $T_i^\pm,T^\pm_j$ are the trace operators provided by Theorem \ref{charactraces} with $f=f_i,f_j$.

Fix a point $p\in S_i\cap S_j$ where $\nu_{S_i}(p)=\pm\nu_{S_j}(p)$; recall that this fact occurs at $\mathscr H^{Q-1}$-a.e. $p\in S_i\cap S_j$. Assume that $\nu_{S_i}(p)=\nu_{S_j}(p)$, i.e., $\frac{Xf_i(p)}{|Xf_i(p)|}=\frac{Xf_j(p)}{|Xf_j(p)|}$; by Theorem \ref{charactraces} we have for $\mathscr H^{Q-1}$-a.e. such $p$ that
	\[
	\begin{aligned}
		|T_i^\pm(p)-T_j^\pm(p)|&=\lim_{r\to 0}\frac{1}{r^Q}\left|\int_{\{\pm f_i>0\}\cap B(p,r)}u\:d\mathscr L^n-\int_{\{\pm f_j>0\}\cap B(p,r)} ud\mathscr L^n\right|\\
		&\leq\lim_{r\to 0}\frac{1}{r^Q}\int_{\{f_if_j\leq 0\}\cap B(p,r)}|u|d\mathscr L^n\\
		&\leq \lim_{r\to 0} \frac{1}{r^Q} \mathscr L^n(\{f_if_j\leq 0\}\cap B(p,r))^{1/Q}  \left(\int_{B(p,r)}|u|^{\frac{Q}{Q-1}}d\mathscr L^n\right)^{\frac{Q-1}{Q}}.
	\end{aligned}
	\] 
By Remark \ref{rmk:germ} we have
	\[
	\lim_{r\to 0} \frac{1}{r^Q}\mathscr L^n(\{f_if_j\leq 0\}\cap B(p,r))=0,
	\]
while by Lemma \ref{negligibleset} we also have that for $\mathscr H^{Q-1}$-almost every $p\in\Omega$ 
	\[
	\limsup_{r\to 0} \frac{1}{r^Q}\int_{B(p,r)}|u|^{\frac{Q}{Q-1}}d\mathscr L^n<+\infty.
	\]
	This proves that $T_i^\pm(p)=T_j^\pm(p)$ for $\mathscr H^{Q-1}$-a.e. $p\in S_i\cap S_j$ such that $\nu_{S_i}(p)=\nu_{S_j}(p)$. A similar argument shows that $T_i^\pm(p)=T_j^\mp(p)$ holds for $\mathscr H^{Q-1}$-a.e. $p\in S_i\cap S_j$ with $\nu_{S_i}(p)=-\nu_{S_j}(p)$. This proves \eqref{eq:equivalenza}, while the last statement of the proposition follows from Theorem \ref{charactraces}.
\end{proof}

The problem of studying  ``intrinsic'' measures of submanifolds of a CC space goes back to M. Gromov \cite[0.6.b]{Gromov}: the interested reader might consult \cite{MagnaniJEMS,MTV,MagV,MSC} and the references therein. Since we do not intend to dwell on such questions, we follow a different (``axiomatic'') path; this is based on the following definition, where we chose to work with the spherical Hausdorff measure $\mathscr S^{Q-1}$, rather than the standard one, because the results mentioned above (as well as \cite{FSSC1,FSSC2}) suggest  $\mathscr S^{Q-1}$ to be more natural than the standard measure $\mathscr H^{Q-1}$.

\begin{definition}[Property \propHD]\label{def:propHD}
Let $(\R^n,X)$ be an equiregular CC space with homogeneous dimension $Q\in\mathbb N$. We say that $(\R^n,X)$ satisfies the {\em property \propHD\ } if  there exists a function $\zeta:\R\times\mathbb S^{m-1}\to(0,+\infty)$ such that, for every $C^1_X$ hypersurface $S\subset\R^n$ and every $p\in S$, one has
\[
\lim_{r\to0}\frac{\mathscr S^{Q-1}(S\cap B(p,r))}{r^{Q-1}}=\zeta(p,\nu_S(p)).
\]
\end{definition}

\begin{rmk}\label{rem:daC1arettif}
If $(\R^n,X)$ is an equiregular CC space satisfying  {\em property \propHD} and $R\subset\R^n$ is $X$-rectifiable, then we have
\[
\lim_{r\to0}\frac{\mathscr S^{Q-1}(R\cap B(p,r))}{r^{Q-1}}=\zeta(p,\nu_R(p))\qquad\text{for $\mathscr S^{Q-1}$-a.e. $p\in R$},
\]
where $\zeta$ is as in Definition \ref{def:propHD}. 

Let us prove this fact. Let $S_i,i\in\mathbb N$, be a family of $C^1_X$ hypersurfaces such that $\mathscr S^{Q-1}(R\setminus\cup_{i\in\mathbb N}S_i)=0$; it is enough to show that, for any fixed $i\in\mathbb N$, we have
\[
\lim_{r\to0}\frac{\mathscr S^{Q-1}(R\cap B(p,r))}{r^{Q-1}}=\zeta(p,\nu_R(p))\qquad\text{for $\mathscr S^{Q-1}$-a.e. $p\in R\cap S_i$}.
\]
Setting $R\Delta S_i\coloneqq(R\setminus S_i)\cup(S_i\setminus R)$, by Remark \ref{densitycorollary} (applied with $\mu\coloneqq\mathscr S^{Q-1}\res(R\Delta S_i)$) we obtain
\[
\lim_{r\to0}\frac{\mathscr S^{Q-1}((R\Delta S_i)\cap B(p,r))}{r^{Q-1}}=0\qquad\text{for $\mathscr S^{Q-1}$-a.e. $p\in R\cap S_i$},
\]
which gives for $\mathscr S^{Q-1}$-a.e. $p\in R\cap S_i$
\[
\lim_{r\to0}\frac{\mathscr S^{Q-1}(R\cap B(p,r))}{r^{Q-1}}=
\lim_{r\to0}\frac{\mathscr S^{Q-1}(S_i\cap B(p,r))}{r^{Q-1}}=
\zeta(p,\nu_{S_i}(p))=
\zeta(p,\nu_R(p))
\]
as desired.
\end{rmk}

Assuming properties $\mathcal R$ and \propHD\ we are able to prove the following result, where we use the notation $u^+_R,u^-_R$ of Proposition \ref{rectifiabletraces}.

\begin{theorem}\label{teo:jumppartHDcompleto}
	Let $(\R^n,X)$ be an equiregular CC space satisfying properties $\mathcal R$ and \propHD; then, there exists a function $\sigma:\R^n\times\mathbb S^{m-1}\to(0,+\infty)$ such that the following holds. For every open set $\Omega\subset\R^n$,  $u\in BV_X(\Omega; \R^k)$ and every countably $X$-rectifiable set $R\subset\R^n$ one has
\[
D_Xu\res R=\sigma(\cdot,\nu_R)(u^+_R-u^-_R)\otimes \nu_R\: \mathscr S^{Q-1}\res R.
\]
In particular, $D^j_Xu=\sigma(\cdot,\nu_u)(u^+-u^-)\otimes \nu_u\: \mathscr S^{Q-1}\res \mathcal J_u$.
\end{theorem}

\begin{proof}
We can assume without loss of generality that $k=1$ and $\mathscr S^{Q-1}(R)<\infty$. By Theorem \ref{federervolpert} and Proposition \ref{prop3.76} we can also assume that $R\subset \mathcal J_u$. Given $p\in\R^n$ we work in adapted exponential coordinates $F_p$ around $p$ and we define
\[
\sigma(p,\nu)\coloneqq\frac{ |\det \nabla F_p(0)|\mathscr H^{n-1}_{e}(\overline \nu^\perp\cap\widehat B_p(0,1))}{\zeta(p,\nu)}
\]
where $\zeta$ is as in Definition \ref{def:propHD} and, as in the proof of Proposition \ref{prop3.76}, $\mathscr H^{n-1}_{e}$ denotes the Euclidean Hausdorff measure in $\R^n$.

Let $\mu_R\coloneqq D_Xu\res R$; by Proposition \ref{prop3.76} we have $\mu_R\ll\mathscr S^{Q-1}\res R$. By Remark \ref{rem:daC1arettif} we can use \cite[Theorem 2.9.8]{Federer} (joint with \cite[Theorem 2.8.17]{Federer}) and it is enough to prove that for $\mathscr S^{Q-1}$-a.e. $p\in R$
\[
\lim_{r\to0} \frac{\mu_R(B(p,r))}{\mathscr S^{Q-1}(R\cap B(p,r))}= \sigma(p,\nu_R(p))(u^+_R(p)-u^-_R(p)) \nu_R(p) ;
\]
notice that the limit above exists $\mathscr S^{Q-1}$-almost everywhere. Taking into account Remark \ref{rem:daC1arettif} and the fact that (by Remark \ref{densitycorollary})  
\[
\lim_{r\to 0}\frac{|D_Xu-\mu_R|(B(p,r))}{r^{Q-1}}=0\qquad\text{for $\mathscr S^{Q-1}$-a.e. $p\in R$},
\]
it suffices to prove that, for $\mathscr S^{Q-1}$-a.e. $p\in R$,  there exists an infinitesimal sequence $(r_i)$ such that
\[
\lim_{i\to+\infty} \frac{D_Xu(B(p,r_i))}{r_i^{Q-1}}= |\det \nabla F_p(0)|\mathscr H^{n-1}_{e}(\overline \nu^\perp\cap\widehat B_p(0,1))(u^+_R(p)-u^-_R(p)) \nu_R(p).
\]
We prove that such a sequence exists at all points where $\limsup_{r\to 0}\frac{|D_Xu|(B(p,r))}{r^{Q-1}}<\infty$, which holds for $\mathscr S^{Q-1}$-a.e. $p\in R$ due to Remark \ref{densitycorollary}.

Let then such a $p\in R$ be fixed; since $R\subset\mathcal J_u$, the functions $\widetilde u_r\coloneqq u\circ F_p\circ\delta_r$ converge in $L^1_{loc}(\R^n)$ to
	\[
	w_p(y)\coloneqq \begin{cases}
	u^+(p) & \text{if }\widetilde L_{\nu_R(p)}(y) \geq 0\\
	u^-(p) & \text{if }\widetilde L_{\nu_R(p)}(y) <0,
	\end{cases}
	\]
	where we used the fact that $\nu_R=\nu_{\mathcal J_u}=\nu_u$  $\mathscr S^{Q-1}$-a.e. on $R$. Let $\widetilde u\coloneqq u\circ F_p$; since (recall notation \eqref{eq:defcampiriscalati}) $|D_{\widetilde X^r}\widetilde u_r|(\widetilde B_r(0,\varrho))=|D_{\widetilde X}\widetilde u|(\widetilde B(0,r\varrho))/r^{Q-1}$ is bounded as $r\to0$ for any positive $\varrho$, by Remark \ref{rem:convergenzadebole}  the sequence $D_{\widetilde X^r}\widetilde u_r$ weakly$^\ast$ converges  in $\R^n$ to $D_{\widehat X}w_p$ as $r\to0$. Let $s_i$ be an infinitesimal sequence such that $|D_{\widetilde X^{s_i}}\widetilde u_{s_i}|$  weakly$^\ast$ to some measure $\lambda$ in $\R^n$; let $\varrho\in (0,1)$ be such that $\lambda(\partial\widehat B_p(0,\varrho))=0$ (which holds for all except at most countably many  $\varrho$) and define $r_i\coloneqq\varrho s_i$.  Proposition \ref{BVdiffeo} gives
\begin{align*}
\lim_{i\to\infty}\frac{D_Xu(B(p,r_i))}{r_i^{Q-1}} = & |\det \nabla F_p(0)|\ \lim_{i\to\infty}\frac{D_{\widetilde X}\widetilde u(\widetilde B(0,r_i))}{r_i^{Q-1}}\\
 = & |\det \nabla F_p(0)|\ \lim_{i\to\infty}\frac{D_{\widetilde X^{s_i}}\widetilde u^{s_i}(\widetilde B_{s_i}(0,\varrho))}{\varrho^{Q-1}}.
\end{align*}
We prove in a moment that
\begin{equation}\label{eq:busillis}
\lim_{i\to\infty}\frac{D_{\widetilde X^{s_i}}\widetilde u^{s_i}(\widetilde B_{s_i}(0,\varrho))}{\varrho^{Q-1}} = \frac{D_{\widehat X} w_p(\widehat B_p(0,\varrho))}{\varrho^{Q-1}};
\end{equation}
assuming this to be true, we have
\begin{align*}
\lim_{i\to\infty}\frac{D_Xu(B(p,r_i))}{r_i^{Q-1}} =& |\det \nabla F_p(0)|\  \frac{D_{\widehat X} w_p(\widehat B_p(0,\varrho))}{\varrho^{Q-1}}\\
=&|\det \nabla F_p(0)|\mathscr H^{n-1}_{e}(\overline \nu^\perp\cap\widehat B_p(0,1))(u^+_R(p)-u^-_R(p)) \nu_R(p).
\end{align*}
and the proof would be concluded.

Let us prove \eqref{eq:busillis}. Defining
\[
\mu_i\coloneqq D_{\widetilde X^{s_i}}\widetilde u^{s_i}\res\widetilde B_{s_i}(0,\varrho),\qquad \mu\coloneqq D_{\widehat X} w_p\res\widehat B_p(0,\varrho)
\]
and taking into account \cite[Proposition 1.62 (b)]{AFP}, it will suffice to show that
\begin{equation}\label{eq:busillis2}
\mu_i\stackrel{\ast}{\rightharpoonup}\mu\qquad\text{and}\qquad |\mu_i|\stackrel{\ast}{\rightharpoonup}\lambda\res \widehat B_p(0,\varrho).
\end{equation}
Concerning the first statement in \eqref{eq:busillis2}, fix a test function $\varphi\in C^0_c(\R^n)$; then
\begin{align*}
&\lim_{i\to\infty} \int\varphi\:d\mu_i
= \lim_{i\to\infty}\int_{\widetilde B_{s_i}(0,\varrho)}\varphi\:dD_{\widetilde X^{s_i}}\widetilde u^{s_i} \\
= &\lim_{i\to\infty}\int_{\widehat B_p(0,\varrho)}\varphi\:dD_{\widetilde X^{s_i}}\widetilde u^{s_i}
+ \int_{\widetilde B_{s_i}(0,\varrho)\setminus \widehat B_p(0,\varrho)}\varphi\:dD_{\widetilde X^{s_i}}\widetilde u^{s_i}
- \int_{\widehat B_p(0,\varrho)\setminus\widetilde B_{s_i}(0,\varrho)}\varphi\:dD_{\widetilde X^{s_i}}\widetilde u^{s_i}\\
=& \lim_{i\to\infty}\int_{\widehat B_p(0,\varrho)}\varphi\:dD_{\widehat X} w_p,
\end{align*}
where the last equality follows from the  weak$^\ast$ convergence of $D_{\widetilde X^{s_i}}\widetilde u^{s_i}$ to $D_{\widehat X} w_p$ and the fact that (denoting by $\Delta$ the symmetric difference of sets)
\[
\lim_{i\to\infty}|D_{\widetilde X^{s_i}}\widetilde u^{s_i}|(\widetilde B_{s_i}(0,\varrho)\Delta \widehat B_p(0,\varrho))=0
\]
that, in turn, can be proved as follows. For any $\varepsilon>0$ there exists $\delta\in(0,\varrho)$ such that
\[
\lambda\big(\overline{\widehat B_p(0,\varrho+\delta)}\setminus \widehat B_p(0,\varrho-\delta)\big)<\varepsilon;
\]
by Theorem \ref{th:localgroup} we obtain
\begin{align*}
\limsup_{i\to\infty}|D_{\widetilde X^{s_i}}\widetilde u^{s_i}|(\widetilde B_{s_i}(0,\varrho)\Delta \widehat B_p(0,\varrho)) 
\leq& \limsup_{i\to\infty}|D_{\widetilde X^{s_i}}\widetilde u^{s_i}|\big(\overline{\widehat B_p(0,\varrho+\delta)}\setminus \widehat B_p(0,\varrho-\delta)\big)\\
\leq& \lambda\big(\overline{\widehat B_p(0,\varrho+\delta)}\setminus \widehat B_p(0,\varrho-\delta)\big)<\varepsilon,
\end{align*}
where we used \cite[Proposition 1.62 (a)]{AFP}. 

The first statement in \eqref{eq:busillis2} is proved; we are left with the second one,  which can be easily proved by  the very same argument taking into account that $|\mu_i|=|D_{\widetilde X^{s_i}}\widetilde u^{s_i}|\res\widetilde B_{s_i}(0,\varrho)$.
\end{proof}

Let us recall once more the notation $u_{p,r}\coloneqq \fint_{B(p,r)}u\:d\mathscr L^n$.

\begin{lemma}\label{circularcrown}
	Let $(\R^n,X)$ be an equiregular  $CC$ space of homogeneous dimension $Q$ and let $\Omega\subseteq\R^n$ be an open bounded set. Then there exist $C=C(\Omega)>0$ and $R=R(\Omega)>0$ such that, for every $p\in \Omega$, every $u\in BV_X(\Omega;\R^k)$ and  every $0<r<\min\{R,\tfrac12d(p,\partial\Omega)\}$, one has
\[
\left|u_{p,2r}-u_{p,r}\right|\leq Cr^{1-Q}|D_Xu|(B(p,2r)).
\]
\end{lemma}
\begin{proof}
We use Theorems \ref{ccproperties} and \ref{teo:poincare} to estimate
\begin{align*}
\left|u_{p,2r}-u_{p,r}\right| = &
\left|\fint_{B(p,r)} (u-u_{p,2r})\:d\mathscr L^n\right|\ \leq \ C \fint_{B(p,2r)}|u-u_{p,2r}|\:d\mathscr L^n\\
\leq & C \left(\fint_{B(p,2r)}|u-u_{p,2r}|^{\frac{Q}{Q-1}}\:d\mathscr L^n\right)^{\frac{Q-1}{Q}}\ 
\leq \ C r^{1-Q} |D_X u|(B(p,2r)).
\end{align*}
\end{proof}

As  mentioned in the Introduction, the next lemma possesses its own interest and it is the key tool in the proof of Theorem \ref{teo:CalderonZ}.

\begin{lemma}\label{lemma3.81}
Let $(\R^n,X)$ be an equiregular $CC$ space of homogeneous dimension $Q$ and let $\Omega\subseteq\R^n$ be an open bounded set. Then there exist $C=C(\Omega)>0$ and $R=R(\Omega)>0$ such that the following holds: for every $u\in BV_X(\Omega;\R^k)$, $p\in \Omega\setminus \mathcal S_u$  and   $0<r<\min\{R,\tfrac12d(p,\partial\Omega)\}$ one has
\[
	\int_{B(p,r)}\frac{|u(q)-u^\star(p)|}{d(p,q)}d\mathscr L^n(q)\leq C\left(|D_Xu|(B(p,r))+\int_0^1\frac{|D_Xu|(B(p,tr))}{t^Q}dt\right).
	\]
	In particular 
	\[
	\int_{B(p,r)}\frac{|u(q)- u^\star(p)|}{d(p,q)}d\mathscr L^n(q)\leq C\int_0^2\frac{|D_Xu|(B(p,tr))}{t^Q}dt.
	\]
\end{lemma}
\begin{proof}
Let $u,p,r$ be as in the statement; we introduce the compact notation $u_i\coloneqq u_{p,2^{-i}r}$, $i\in\mathbb N$. Since $u_i\to u^\star (p)$ as $i\to\infty$ we estimate
\begin{align*}
	& \int_{B(p,r)}\frac{|u(q)-u^\star(p)|}{d(p,q)}d\mathscr L^n(q)\\
	\leq & \sum_{i=1}^\infty \int_{B(p,2^{-i+1}r)\setminus B(p,2^{-i}r)}\frac{|u(q)-{u}^\star(p)|}{2^{-i}r}d\mathscr L^n(q)\\
	\leq&\sum_{i=1}^\infty\frac{2^i}{r}\int_{B(p,2^{-i+1}r)\setminus B(p,2^{-i}r)}\bigg(|u(q)-u_{i-1}| + \sum_{j=i-1}^\infty|u_j-u_{j+1}| \bigg)d\mathscr L^n(q)
\intertext{and use Lemma \ref{circularcrown} and Theorem \ref{teo:poincare} to get}
	\leq & C\sum_{i=1}^\infty\frac{2^i}{r}\left(2^{-i}r|D_Xu|(B(p,2^{1-i}r))+
	\sum_{j=i-1}^\infty \left(2^{1-i}r\right)^Q\left(2^{-(j+1)}r\right)^{1-Q}|D_Xu|(B(p,2^{-j}r))\right)\\
	\leq& C\sum_{i=1}^\infty\left(|D_Xu|(B(p,2^{1-i}r))+\sum_{j=i-1}^\infty2^{(j-i+1)(Q-1)}|D_Xu|(B(p,2^{-j}r))\right)\\
	=&C\sum_{k=0}^\infty \Big(1+1+2^{Q-1}+(2^{Q-1})^2+\dots+ (2^{Q-1})^k\Big)|D_Xu|(B(p,2^{-k}r))	\\
	\leq&C\sum_{k=0}^\infty\frac{2^{(k+1)(Q-1)}-1}{2^{Q-1}-1}|D_Xu|(B(p,2^{-k}r)). 
	\end{align*}
	Since $Q\geq 2$ we have $2^{Q-1}-1\geq \tfrac{2^{Q-1}}{2}$, hence
	\[
	\begin{aligned}
	\int_{B(p,r)}\frac{|u(q)- u^\star(p)|}{d(p,q)}d\mathscr L^n(q)&\leq C\sum_{k=0}^\infty 2^{k(Q-1)}|D_Xu|(B(p,2^{-k}r))\\
	&= C\left(|D_Xu|(B(p,r))+\sum_{k=1}^\infty2^{k(Q-1)}|D_Xu|(B(p,2^{-k}r))\right)\\
	&=C\left(|D_Xu|(B(p,r))+\sum_{k=1}^\infty\int_{2^{-k}}^{2^{1-k}}2^{kQ}|D_Xu|(B(p,2^{-k}r))dt\right)\\
	&\leq C\left(|D_Xu|(B(p,r))+\sum_{k=1}^\infty\int_{2^{-k}}^{2^{1-k}}\frac{|D_Xu|(B(p,tr))}{t^Q}dt\right)\\
	&=C\left(|D_Xu|(B(p,r))+\int_0^1\frac{|D_Xu|(B(p,tr))}{t^Q}dt\right),
	\end{aligned}
	\]
as desired.
\end{proof}

We can now prove one of our main results; recall that we denote by $D_X^{ap}u(p)$ the approximate $X$-gradient of $u$ at $p$.

\begin{proof}[Proof of Theorem \ref{teo:CalderonZ}]
We can assume without loss of generality that $k=1$. Suppose that $D_Xu=v \mathscr L^n+ D_X^su$ is the Radon-Nykod\'ym decomposition of the measure $D_Xu$ with respect to $\mathscr L^n$. By the Radon-Nykod\'ym  Theorem in doubling metric spaces (see e.g. \cite[Theorem 4.7 and Remark 4.5]{Simon}), at $\mathscr L^n$-almost every $p\in \Omega$ we have
	\begin{equation}\label{eq:singular}
	\lim_{r\to 0} \frac{D^s_Xu(B(p,r))}{r^Q}=0.
	\end{equation}
	It is sufficient to prove that, for every $p\in \Omega\setminus(\mathcal S_u\cup\mathcal S_v)$ for which \eqref{eq:singular}  holds, $u$ is approximately $X$-differentiable at $p$ with $D^{ap}_Xu(p)=v^\star(p)$. 
	
	Let $R>0$ and $f\in C^1(B(p,R))$ be such that $f(p)=0$ and $Xf(p)=v^\star(p)$ and define 
	\[
	w(q)\coloneqq u(q)- u^\star(p)-f(q)
	\]
	Then $w \in BV_X(B(p,R))$, $p\notin \mathcal S_w$ and $ w^\star(p)=0$. We are in a position to apply Lemma \ref{lemma3.81} to the function $w$ and get $C>0$ so that, for small enough $r$,
	\[
	\begin{aligned}
	\frac{1}{r^Q}\int_{ B(p,r)}\frac{\left| u(q)- u^\star(p)- f(q) \right|}{d(p,q)}d\mathscr L^n(q)&\leq \frac{C}{r^Q}\int_0^2\frac{|D_Xw|(B(p,tr))}{t^Q}dt\\
	&\leq C\sup_{t\in (0,2)}\frac{|D_Xw|(B(p,tr))}{(tr)^Q}.
	\end{aligned}
	\]
	It is then enough to show that $\lim_{r\to 0}r^{-Q}|D_Xw|(B(p,r))=0$.     Taking into account that $D_Xw=(v-Xf)\mathscr L^n +D_X^su$ and \eqref{eq:singular}, it suffices to check that
    \[
    \lim_{r\to 0}\frac 1{r^Q}\int_{B(p,r)}|v-Xf|d\mathscr L^n=0,
    \]
    which follows by the generalized Lebesgue's differentiation theorem (see e.g. \cite[Section 2.7]{Heinonen}) and the inequality $|v-Xf|\leq |v-v^\star(p)|+|v^\star(p)-Xf|$.
\end{proof}

As for classical $BV$ functions (see e.g. \cite[pag. 177]{AFP}, the (approximate) convergence of $u\in BV_X$ to $u^\star(p)$ at points $p\notin\mathcal S_u$ can be improved in a $L^{1^\ast}$-sense, as we now state.

\begin{proposition}\label{prop:medie1star}
Let $(\R^n,X)$ be an equiregular CC space, $\Omega\subset\R^n$ an open set and let $u\in BV_X(\Omega)$. Then
  \[
  \lim_{r\to 0}\fint_{B(p,r)}|u-u^\star(p)|^{\frac{Q}{Q-1}}d\mathscr L^n=0\qquad\text{for $\mathscr H^{Q-1}$-a.e. }p\in\Omega\setminus\mathcal S_u.
  \]
\end{proposition}
\begin{proof}
We first prove that 
  \begin{equation}\label{eq:orQ-1}
  \lim_{r\to 0}\frac{|D_Xu|(B(p,r))}{r^{Q-1}}=0\qquad\text{for $\mathscr H^{Q-1}$-a.e. }p\in\Omega\setminus\mathcal S_u.
  \end{equation}
Let $t>0$ be fixed and consider the set
\[
E_t\coloneqq \left\{ p\in\Omega\setminus\mathcal S_u:\limsup_{r\to0}\frac{|D_Xu|(B(p,r))}{r^{Q-1}}>t\right\}.
\]
By Proposition \ref{k-density} one has $\mathscr H^{Q-1}(E_t)<\infty$; Proposition \ref{prop3.76} then implies that $|D_Xu|(E_t)=0$ and again Proposition \ref{k-density} gives $\mathscr H^{Q-1}(E_t)=0$. Since this is true for all positive $t$, formula \eqref{eq:orQ-1} immediately follows.

  Combining Theorem \ref{teo:poincare} and \eqref{eq:orQ-1} we immediately get that for $\mathscr H^{Q-1}$-a.e. $p\in\Omega$
  \[
 \lim_{r\to 0}\fint_{B(p,r)}|u-u_{p,r}|^{\frac{Q}{Q-1}}d\mathscr L^n=0.
 \]
The conclusion follows by 
\[
|u-u^\star(p)|^{\frac{Q}{Q-1}}\leq 2^{\frac {1}{Q-1}}\left(|u_{p,r}-u^\star(p)|^{\frac{Q}{Q-1}}+|u-u_{p,r}|^{\frac{Q}{Q-1}}\right).
\]
together with $u^\star(p)=\lim_{r\to0}u_{p,r}$.
\end{proof}

When $(\R^n,X)$ satisfies property $\mathcal R$, $\Omega\subset\R^n$ is open and  $u\in BV_{X}(\Omega,\R^k)$, by Theorem \ref{federervolpert} the {\em precise representative} $u^\preciso$ 
\begin{equation}\label{eq:defrapprespreciso}
u^\preciso(p)\coloneqq
\begin{cases}
u^\star(p)& \text{if }p\in\Omega\setminus\mathcal S_u\vspace{.2cm}\\
\dfrac{u^+(p)+u^-(p)}{2}\quad& \text{if }p\in \mathcal J_u
\end{cases}
\end{equation}
is defined $\mathscr H^{Q-1}$-a.e. on $\Omega$. We have the following result.

\begin{theorem}\label{teo:convergenzaalrapprpreciso}
	Let $(\R^n,X)$ be an equiregular CC space satisfying property $\mathcal R$, $\Omega\subset\R^n$ an open set and let $u\in BV_X(\Omega; \R^k)$. Then
	\[
\lim_{r\to 0}\ \fint_{B(p,r)}u\:d\mathscr L^n=u^\preciso(p)\qquad\text{for $\mathscr H^{Q-1}$-a.e. }p\in\Omega.	
	\]
\end{theorem}
\begin{proof}
The statement easily follows for $\mathscr H^{Q-1}$-a.e. $p\in\Omega\setminus\mathcal S_u$ by Proposition \ref{prop:medie1star}. By Theorem \ref{federervolpert} it suffices to prove the statement for all $p\in \mathcal J_u$, which directly follows from   Proposition \ref{goodsplit} and Definition \ref{approximatejump}.
\end{proof}

\begin{rmk}
When $(\R^n,X)$ satisfies property $\mathcal R$, then $D^c_Xu=D^s_Xu\res(\Omega\setminus \mathcal S_u)$: to see this, it is enough to combine Proposition  \ref{prop3.76} and Theorem \ref{federervolpert}.
\end{rmk}

We now want to study the properties of the decomposition $D_Xu=D_X^au+D_X^cu+D^j_Xu$; recall that $\mathscr H^{1}_{e}$ denotes the Euclidean Hausdorff measure in $\R^n$.

\begin{theorem}[Properties of Cantor part and jump part]\label{teo:proprieta,a,c,j}
	Let $u\in BV_X(\Omega;\R^k)$. Then the following facts hold:
	\begin{itemize}
		\item[(a)] $D^a_Xu=D_Xu\res(\Omega\setminus S)$ and $D^s_Xu=D_Xu\res S$, where
			\[
			S\coloneqq \left\{p\in \Omega: \lim_{r\to 0}\frac{|D_Xu|(B(p,r))}{r^Q}=+\infty\right\}.
			\]
Moreover, if $E\subset\R^k$ is such that $\mathscr H^1_e(E)=0$, then $D^{ap}_Xu=0$ $\mathscr L^n$-a.e. in $(u^\star)^{-1}(E)$.
		\item[(b)] Let $\Theta_u\subset S$ be defined by
			\[
			\Theta_u\coloneqq\left\{p\in \Omega: L(p)\coloneqq\liminf_{r\to 0}\frac{|D_Xu|(B(p,r))}{r^{Q-1}}>0\right\}.
			\]
			Then $\mathcal J_u\subseteq \Theta_u$. 		
	\end{itemize}
Moreover, if $(\R^n,X)$ satisfies property $\mathcal R$, then
\begin{itemize}
\item[(c)]  $\mathscr H^{Q-1}(\Theta_u\setminus \mathcal J_u)=0$ and $D^j_Xu=D_Xu\res\Theta_u$. More generally, for every Borel set $\Sigma$ containing $\mathcal J_u$ and $\sigma$-finite with respect to $\mathscr H ^{Q-1}$ we have $D_X^ju=D_Xu\res \Sigma$.
\item[(d)] $D_X^cu=D_Xu\res(S\setminus \Theta_u)$. 
\item[(e)] if  $B\subset\Omega$ is such that either $\mathscr H^{Q-1}\res B$ is $\sigma$-finite or $B=(u^\star)^{\,-1}(E)$ for some $\mathscr H^1_e$-negligible set $E\subseteq \R^k$, then $D_X^cu(B)=0$.
\end{itemize}	
\end{theorem}
\begin{proof}
	In order to prove the first part of statement (a) it is sufficient to apply Radon-Nykod\'ym  Theorem in doubling metric spaces (see e.g. \cite[Theorem 4.7 and Remark 4.5]{Simon}). Concerning the second part, assume first that $k=1$ and let $B\coloneqq(u^\star)^{-1}(E)$. By Proposition \ref{levelsets}, for any $t\notin E$ we have $	B\cap\partial^\ast\{u>t\}=\emptyset$. By Theorems \ref{coarea} and \ref{teo:ambrosio} we obtain
\[
|D_Xu|(B)=\int_\R P_X(\{u>t\}\cap B)\:dt=0= \int_{\R\setminus E} \int_{\partial^\ast\{u>t\}\cap B}\theta_td\mathscr H^{Q-1}\:dt=0,
\]
where $\theta_t$ denote suitable positive functions. When $k\geq 1$ and $j=1,\dots,k$ we set $E_j\coloneqq\{t\in\R:t=z_j\text{ for some }z\in E\}$; the set $E_j$ is such that $\mathscr L^1(E_j)=0$ and by \eqref{vectorbv} 
\[
|D_Xu|(B)\leq \sum_{j=1}^k |D_X u^j|(B) \leq \sum_{j=1}^k |D_X u^j|(((u^j)^\star)^{-1}(E_j))=0.
\]
We then conclude by Theorem \ref{teo:CalderonZ}.
	
By \eqref{inequality3.76} in the proof of Proposition \ref{prop3.76} we have $\mathcal J_u\subseteq \Theta_u$, and statement (b) follows. 
	
	We now prove (c). Applying Proposition \ref{k-density} we get that for every $h\in \mathbb N\setminus\{0\}$
	\begin{equation}\label{eq:sonno}
	|D_Xu|\res\{L\geq \tfrac{1}{h}\}\geq \frac 1h \omega_{Q-1}\mathscr H^{Q-1}\res \{L\geq \tfrac{1}{h}\},
	\end{equation}
where $L$ is defined in statement (b). In particular $\mathscr H^{Q-1}\left(\{L\geq \tfrac{1}{h}\}\right)<+\infty$. By \eqref{2implication}  
	\[
	|D_Xu|\left(\{L\geq\tfrac{1}{h}\}\setminus \mathcal S_u\right)=0
	\]
	and consequently (by \eqref{eq:sonno}) also $\mathscr H^{Q-1}(\{L\geq\tfrac{1}{h}\}\setminus \mathcal S_u)=0$. Since $\{L\geq\tfrac{1}{h}\}\nearrow\Theta_u$, on passing to the limit for $h\to +\infty$ we get $\mathscr H^{Q-1}(\Theta_u\setminus \mathcal S_u)=0$. Taking Theorem \ref{federervolpert} into account, we conclude that $\mathscr H^{Q-1}(\Theta_u\setminus \mathcal J_u)=0$.\\
	Let now $\Sigma$ be as in  statement (c). Then, taking into account Proposition \ref{prop3.76} and the fact that $\mathscr H ^{Q-1}(\mathcal S_u\setminus \mathcal J_u)=0$, we have
	\[
	\begin{aligned}
	D_Xu \res \Sigma&=D_Xu\res\mathcal J_u+D_Xu\res (\Sigma\setminus \mathcal J_u)\\
	&=D_X^ju+D_Xu\res (\Sigma\setminus \mathcal S_u)+D_Xu\res (\Sigma\cap \mathcal S_u \setminus \mathcal J_u)\\
	&=D_X^ju+D_Xu\res (\Sigma\setminus \mathcal S_u).
	\end{aligned}
	\]
	Since $\Sigma$ is $\sigma$-finite with respect to $\mathscr H^{Q-1}$, using \eqref{2implication} we get that $D_Xu\res (\Sigma\setminus \mathcal S_u)=0$, and so $D_Xu\res \Sigma=D_X^ju$.
	
	Statement $(d)$ follows from (a), (b), (c) and  the decomposition $D_Xu=D^a_Xu+D^c_Xu+D_X^ju$, which immediately give that $D_X^cu=D_Xu\res (S\setminus \Theta_u)$. 
	
	We prove (e) in case  $\mathscr H^{Q-1}\res B$ is $\sigma$-finite; we can assume (see e.g. \cite[Theorem 1.43]{AFP}) that $B$ is a Borel set. Using Proposition \ref{prop3.76} and Theorem \ref{federervolpert} we get that $|D_Xu|(B\setminus \mathcal J_u)=0$, which  gives $(D _X^au+D^c_Xu) \res B=0$.\\
	Concerning the second part of statement $(e)$, suppose first that $k=1$ and let $B=(u^\star)^{\,-1}(E)$ with $\mathscr L^1(E)=0$. By Proposition \ref{levelsets} we know that $\partial^*\{u>t\}\cap B=\emptyset$ for every $t\notin E$. Applying the Coarea Formula of Theorem \ref{coarea} we get
	\[
	|D_Xu|(B)=\int _E \int_{\partial ^*\{u>t\}\cap B}\theta_td\mathscr H^{Q-1}dt=0
	\]
	for suitable functions $\theta_t$. 	In the general case $k\geq 2$ define for every $\alpha=1,\dots, k$ the sets $
	E_\alpha\coloneqq\pi_\alpha(E)$, where $\pi_\alpha$ denotes the canonical projection $\pi_\alpha(x_1,\dots,x_k)=x_\alpha$. Noticing that $\mathscr L^1(E_\alpha)\leq \mathscr H^1_e(E)=0$, we can use \eqref{vectorbv} to estimate
	\[
	|D_Xu|( (u^\star)^{\,-1}(E))\leq \sum_{\alpha=1}^k|D_Xu^\alpha|((u^\star)^{\,-1}(E))\leq \sum_{\alpha=1}^k |D_Xu^\alpha|(((u^\alpha)^\star)^{\,-1}(E_\alpha))=0.
	\]
\end{proof}

\section{Applications to some classes of Carnot groups}\label{sec:proprietaLR}
Some of the main results of this paper rely on properties $\mathcal R,\mathcal{LR}$ or \propHD; in this section we show how they can be in some meaningful  CC spaces and, in particular, in some large classes of Carnot groups.

We start by introducing the reduced boundary $\mathcal F_XE$ of a set $E$ with finite $X$-perimeter. Recall that the reduced boundary was the object originally considered by E. De Giorgi in the seminal paper \cite{Deg55} about the rectifiability of sets with finite (Euclidean) perimeter in $\R^n$.

\begin{definition}[Reduced boundary]
Let $E\subset \R^n$ be a set with locally finite $X$-perimeter. The {\em $X$-reduced boundary} $\mathcal F_XE$ of $E$ is the set  of points  $p\in\R^n$ such that $P_X(E,B(p,r))>0$ for any $r>0$ and the limit
  \[
\widetilde\nu_E(p):=  \lim_{r\to 0}\frac {D_X\chi_E(B(p,r))}{|D_X\chi_E|(B(p,r))}
  \]
exists with $|\widetilde\nu_E(p)|=1$.
\end{definition}

For sets with finite (Euclidean) perimeter in $\R^n$  the symmetric difference between the essential boundary and the reduced one is $\mathscr H^{n-1}_e$-negligible, see e.g. \cite[Theorem 3.61]{AFP}. In our setting we have the following result, which is a known consequence of Theorem \ref{teo:ambrosio}, see e.g. \cite[Theorem 7.3]{FSSC1} for the {\em Heisenberg group} case and \cite[Lemma 2.26]{FSSC2} for step 2 Carnot groups.

\begin{theorem}\label{th:essenzialeridotta}
  Let $(\R^n, X)$ be an equiregular CC space of homogeneous dimension $Q$ and let $E\subseteq\R^n$ be a set of locally finite $X$-perimeter. Then $\mathscr H^{Q-1}(\partial^\ast E\setminus\mathcal F_XE)=0$.
\end{theorem}
\begin{proof}
  By Theorem \ref{teo:ambrosio} we have $ D_X\chi_E=\theta\nu_E \mathscr H^{Q-1}\res \partial^\ast E$ for a suitable positive function $\theta$. Therefore it is enough to prove that, for $\mathscr H^{Q-1}$-almost every $p\in \partial^\ast E$, one has
\[
\lim_{r\to 0}\frac{D_X \chi_E(B(p,r))}{|D\chi_E|(B(p,r))}=\nu_E(p).
\]
This fact directly follows from \cite[Theorem 2.9.8]{Federer} taking into account Theorem \ref{teo:ambrosio} and \cite[Theorem 2.8.17]{Federer}.
\end{proof}

The proof of Theorem \ref{th:essenzialeridotta} also shows that $\widetilde\nu_E=\nu_E$ $\mathscr H^{Q-1}$-a.e. on $\mathcal F_XE$. 

The papers \cite{FSSC1,FSSC2,Marchi} prove  the countable $X$-rectifiability of the reduced boundary of sets with locally finite $X$-perimeter in, respectively, Heisenberg groups,  Carnot groups of step 2, and Carnot groups {\em of type $\star$}. These results, in conjunction with Theorem \ref{th:essenzialeridotta}, show that property $\mathcal{R}$ is satisfied in these settings.

Actually,  Theorem \ref{th:essenzialeridotta} and the results about blow-up and representation of the $X$-perimeter available in Heisenberg groups (\cite[Theorems 4.1 and 7.1]{FSSC1}), step 2 Carnot groups (\cite[Theorems 3.1 and 3.9]{FSSC2}) and Carnot groups of type $\star$ \cite[Theorems 4.12 and 4.13]{Marchi} imply that also   property \propHD\ is satisfied in these settings. 

Using also the left-invariance of the structure we can  conclude what follows.

\begin{theorem}
Heisenberg groups, Carnot groups of step 2 and Carnot groups of type $\star$ satisfy properties $\mathcal{R}$ and \propHD. In particular,  Theorems \ref{federervolpert}, \ref{teo:decomp-intro-2propR}, \ref{teo:jumppartHDintro} and \ref{teo:convergenzaalrapprpreciso} hold in these settings.

Moreover, the function $\sigma(p,\nu)$ appearing in \ref{teo:jumppartHDintro} and \ref{teo:convergenzaalrapprpreciso} does not depend on the point $p\in\R^n$.
\end{theorem}

In the paper \cite{DLDMV} the class of Carnot groups $\mathbb G$ satisfying the following assumption (see e.g. \cite{Montgomery} for the notion of {\em abnormal curve})
\begin{equation}\label{eq:defhyva}
\parbox{14cm}{there exists at least one direction $V$ in the first layer of the stratified Lie algebra of $\mathbb G$ such that  $t\mapsto\exp(tV)$ is not an  abnormal curve}
\end{equation}
is considered.  This class includes, for instance, the {\em Engel group}, which is the simplest example where the rectifiability problem for sets with finite $X$-perimeter is open. One of the main results of \cite{DLDMV} is the following one: for any set $E$ with finite $X$-perimeter in a Carnot group $\mathbb G$ satisfying \eqref{eq:defhyva}, the reduced boundary $\mathcal F_XE$ is countably $X$-Lipschitz rectifiable. Together with Theorem \ref{th:essenzialeridotta}, this  gives the following result.

\begin{theorem}
The property $\mathcal{LR}$ is satisfied in all Carnot groups $\mathbb G$ such that \eqref{eq:defhyva} holds; in particular, Theorem \ref{teo:SuLR} holds in such groups.
\end{theorem}

\appendix
\section{Some tools from Geometric Measure Theory in metric spaces.}\label{app:A}
\begin{proposition}\label{levelsetsbis}
  Let $u\in L^1_{loc}(\Omega; \R^k)$. If $p\in \Omega\setminus \mathcal S_u$, then, for any $\varepsilon>0$, the set
  \[
  E_\varepsilon\coloneqq\{q\in \Omega: |u(q)-u^\star(p)|>\varepsilon\}
  \]
  has density $0$ at $p$. Conversely, if $u\in L^\infty_{loc}(\Omega; \R^k)$ and  $z\in \R^k$ are such that, for any $\varepsilon>0$, the set
  \[
  E_\varepsilon\coloneqq\{q\in \Omega: |u(q)-z|>\varepsilon\}
  \]
  has density $0$ at $p$, then $p\in \Omega\setminus \mathcal S_u$ and $z=u^\star(p)$.

  In particular, if $k=1$ and $p\in \Omega\setminus \mathcal S_u$ and $t\neq u^\star(p)$, then $p\notin \partial^*\{q\in\Omega:u(q)>t\}$.  
\end{proposition}
\begin{proof}
  Suppose $p\in \Omega\setminus \mathcal S_u$. By Chebychev inequality we have
  \[
  \varepsilon \frac{\mathscr L^n(E_\varepsilon\cap B(p,r))}{\mathscr L^n(B(p,r))}\leq \fint_{B(p,r)}|u-u^\star(p)|d\mathscr L^n\to0\qquad\text{as }r\to 0.
  \]

  Conversely, suppose that $u$ and $z$ are as in the statement. Then we have for any $r\in (0,1)$
  \[
  \fint_{B(p,r)}|u-z|d\mathscr L^n\leq (\|u\|_{L^\infty(B(p,1);\R^k)}+|z|)\frac{\mathscr L^n(B(p,r)\cap E_\varepsilon))}{\mathscr L^n(B(p,r))}+\varepsilon \frac{\mathscr L^n(B(p,r)\setminus E_\varepsilon))}{\mathscr L^n(B(p,r))},
  \]
which is infinitesimal as $r\to0$.  
  
  Finally, consider $p\in \Omega\setminus \mathcal S_u$ and let $t\neq u^\star(p)$. We already know that both $\{u>u^\star(p)+\varepsilon\}$ and $\{u<u^\star(p)-\varepsilon\}$ have density $0$ at $p$ for every $\varepsilon>0$. If $t>u^\star(p)$, then choosing $\varepsilon=t-u^\star(p)$ we have that $\{u>t\}$ has density $0$ at $p$. If $t<u^\star(p)$ then choose $\eta>0$ such that  $\varepsilon=u^\star(p)-t-\eta>0$ to infer that $\{u<t+\eta\}$ has density $0$ at $p$, and consequently $\{u\geq t+\eta\}$ has density $1$ at $p$. This implies that also $\{u>t\}$ has density $1$ at $p$.
\end{proof}

The following result is classical, see e.g. \cite{Simon} or  \cite{Heinonen}.

\begin{theorem}[5r-Covering Lemma] \label{5rcovering}
	Let $(M,d)$ be a separable metric space and let $\mathcal B$ a family of closed balls in $M$ such that 
	\[
	\sup\left\{\diam B: B\in \mathcal B\right\}<+\infty.
	\]
	Denote by $5B$ the closed metric ball with  same center as $B$ and radius $5$ times larger than that of $B$. Then there exists a countable and pairwise disjoint subfamily $\mathcal F\subseteq \mathcal B$ such that
	\[
	\bigcup \mathcal B\subseteq \bigcup_{B\in \mathcal F} 5B.
	\]
\end{theorem}

\begin{definition}[Hausdorff measures]\label{def:Hausmeas}
  Let $(M,d)$ be a metric space and $k\geq 0$. For any $\delta>0$ and  any  $E\subset M$ we define
  \begin{align*}
  & \mathscr H^k_\delta(E)\coloneqq\frac{\omega_k}{2^k}\inf\left\{\sum_{h=0}^\infty(\diam E_h)^k: E\subseteq \bigcup_{h=0}^\infty E_h, \diam E_h<\delta\right\}\\
  & \mathscr S^k_\delta(E)\coloneqq\frac{\omega_k}{2^k}\inf\left\{\sum_{h=0}^\infty(\diam B_h)^k: E\subseteq \bigcup_{h=0}^\infty B_h, B_h \text{ balls with } \diam B_h<\delta\right\},
\end{align*}
where $\omega_\alpha\coloneqq\pi^{\alpha/2}\Gamma(1+\alpha/2)^{-1}$ and $\Gamma(t)\coloneqq\int_0^{+\infty}s^{t-1}e^{-s}ds$ is the Euler $\Gamma$ function.  The {\em Hausdorff measure} $\mathscr H^k(E)$ and the {\em spherical Hausdorff measure} $\mathscr S^k(E)$ of  $E$ are   
  \begin{align*}
&\mathscr H^k(E)\coloneqq\sup_{\delta>0}\mathscr H_\delta ^k(E)=\lim_{\delta \to 0}\mathscr H_\delta^k(E)\\
& \mathscr S^k(E)\coloneqq\sup_{\delta>0} \mathscr S_\delta^k(E)=\lim_{\delta \to 0} \mathscr S_\delta^k(E).
\end{align*}
\end{definition}

The useful inequalities $\mathscr H^k \leq \mathscr S^k \leq 2^k\mathscr H^k$ are classical.

If $(M,d,\mu)$ is a  metric measure space, $k\geq 0$ and $x\in M$, we define the \emph{upper $k$-density} $\Theta_k^*(\mu,x)$ and the \emph{lower $k$-density} $\Theta_{*k}(\mu,x)$ of $\mu$ at $x$ as
	\[
	\begin{split}
	&\Theta_k^*(\mu,x)\coloneqq\limsup_{r\to 0}\frac{\mu(B(x,r))}{\omega_k r^k},
	\qquad
	\Theta_{*k}(\mu,x)\coloneqq\liminf_{r\to 0}\frac{\mu(B(x,r))}{\omega_k r^k}.
	\end{split}
	\]
	For every Borel set $E\subseteq \R^n$ we will also write $\Theta_k^*(E,x)\coloneqq\Theta_k^*(\mathscr H^k\res E, x)$ and $\Theta_{*k}(E,x)\coloneqq\Theta_{*k}(\mathscr H^k\res E,x)$. 	If   $\Theta_k^*(\mu,x)=\Theta_{*k}(\mu,x)$, then the common value is denoted by $\Theta_k(\mu,x)$ and it will be called \emph{$k$-density} of $\mu$ at $x$.  Hausdorff measures  and densities are linked by Propositions \ref{k-density} and \ref{k-densityf} below. A proof of Proposition \ref{k-density} can be found for instance in \cite[Theorem 3.2]{Simon};  in the latter reference, statement (i) below is stated with $\mathscr H^k$ in place of $\mathscr S^k$, but the careful reader will notice that the proof is indeed provided for this stronger version.
	
\begin{proposition}\label{k-density}
	Let $(M,d)$ be a separable metric space, let $\mu$ be a Borel regular Radon measure on $M$, let $E\subseteq M$ be a Borel set and let $t>0$. Then the following facts hold.
	\begin{itemize}
		\item[(i)] If $\Theta^*_k(\mu,x)\geq t$ for every $x\in E$, then $\mu\geq t\mathscr S^k\res E$.
		\item[(ii)] If $\Theta^*_k(\mu,x)\leq t$ for every $x\in E$, then $\mu\leq 2^k t\mathscr H^k\res E$.
	\end{itemize} 
	In particular, for $\mathscr H^k$-almost every $x\in \R^n$ we have $\Theta^*_k(\mu,x)<+\infty$.
\end{proposition}

\begin{corollary}\label{k-densityf}
  Let $(M,d)$ be a separable metric space, let $\mu$ be a Borel regular Radon measure on $M$, let $E\subseteq M$ be a Borel set and let $f:E \rightarrow \R$ be a strictly positive function. Then the following facts hold.
  \begin{itemize}
		\item[(i)] If $\Theta^*_k(\mu,x)\geq f(x)$ for every $x\in E$ , then $\mu\geq f\mathscr S^k\res E$.
		\item[(ii)] If $\Theta^*_k(\mu,x)\leq f(x)$ for every $x\in E$ , then $\mu\leq 2^k f\mathscr H^k\res E$.
	\end{itemize}
\end{corollary}
\begin{proof}
(i)  Let $\varepsilon>0$ and define for every $j\in \mathbb Z$ the set
  \[
  E_j\coloneqq \{x\in E: (1+\varepsilon)^j<f(x)\leq (1+\varepsilon)^{j+1}\}.
  \]
  Suppose that $\Theta^*_k(\mu,x)\geq f(x)$ for every $x\in E$. Then, using (i) of Proposition \ref{k-density} we get
  \[
    \mu = \sum_{j\in \mathbb Z} \mu\res E_j\geq \sum_{j\in \mathbb Z} (1+\varepsilon)^j \mathscr S^k \res E_j\geq \sum_{j\in \mathbb Z} \frac f{1+\varepsilon} \mathscr S^k\res E_j= \frac f{1+\varepsilon} \mathscr S^k\res E.
  \]
The statement follows by the arbitrariness of $\varepsilon$.
  
(ii)  Using (ii) of Proposition \ref{k-density} we have
  \[
  \begin{aligned}
    \mu &= \sum_{j\in \mathbb Z} \mu\res E_j\leq \sum_{j\in \mathbb Z} 2^k(1+\varepsilon)^{j+1} \mathscr H^k \res E_j\\
    &\leq \sum_{j\in \mathbb Z} 2^k(1+\varepsilon)f \mathscr S^k\res E_j= 2^k(1+\varepsilon)f \mathscr S^k\res E.
  \end{aligned}
  \]
The statement follows by the arbitrariness of $\varepsilon$.
\end{proof}

As a consequence of the Corollary \ref{k-densityf} we have the following remark.

\begin{rmk}\label{densitycorollary}
  Under the same assumptions of Corollary \ref{k-densityf}, for $\mathscr H^k$-almost every $x\in \R^n$ we have $\Theta_k^*(\mu,x)<+\infty$ and for any Borel set $B\subseteq \R^n$ the implication
\[
  \mu(B)=0\quad\Longrightarrow\quad \Theta_k(\mu,x)=0 \text{ for $\mathscr H^k$-a.e. $x\in B$}
  \]
  holds.  In particular, if $\mu= g\mathscr H^k\res E$ we have $  \Theta_k(\mu,x)=0$  for $\mathscr H^k$-almost every $x\in \R^n\setminus E$.
\end{rmk}

\begin{definition}[Porous sets]\label{def:porous}
Let $(M,d)$ be a metric space and let $E\subseteq M$ be a Borel set. Then $E$ is said to be \emph{porous} if there esist $\alpha\in (0,1)$ and $R>0$ such that for every $x\in M$ and  every $r\in (0,R)$  there exists $y\in M$ such that 
\[
B(y,\alpha r)\subseteq B(x,r)\setminus E.
\] 
\end{definition}
\begin{proposition}\label{th:porousnegligible}
Let $(M,d)$ be a locally compact and separable metric space, $\mu$ a  doubling Radon measure on $M$ and let $E\subseteq M$ be a porous set. Then $E$ has no points of density 1 and, in particular, $\mu(E)=0$.
\end{proposition}
\begin{proof}
Let $\alpha$ and $R$ be as in Definition \ref{def:porous}. Suppose by contradiction there exists $x\in E^1$. For every  $r\in (0,R)$ there exists $y\in M$ such that $B(y,\alpha r)\subseteq B(x,r)\setminus E$. This implies that
\[
\frac{\mu(B(x,r)\setminus E)}{\mu(B(x,r))}
\geq
\frac{\mu(B(y,\alpha r))}{\mu(B(x,r))}
\geq C,
\]
where $C>0$ depends on $\alpha$ and the doubling constant of $\mu$. Letting $r\to 0$ and taking into account that $x\in (M\setminus E)^0$, we get a contradiction. The last part of the statement follows from the generalized Lebesgue Theorem, see e.g. \cite[Theorem 1.8]{Heinonen}.
\end{proof}

\begin{proposition}\label{stella}
  Let $(\R^n, X)$ be a geodesic equiregular CC space; then, for every $p\in \R^n$ and for every $r>0$ one has $\mathscr L^n (\partial B(p,r))=0$.
\end{proposition}
\begin{proof}
By Proposition \ref{th:porousnegligible} it is sufficient to prove that $\partial B(p,r)$ is a porous set. Take $q\in \partial B(p,r)$ and consider a length minimizing absolutely continuous path  $\gamma:[0,r]\rightarrow \R^n$ such that $\gamma(0)=p$, $\gamma(r)=q$ and for every $t\in [0,r]$ one has $d(p,\gamma(t))=t$. Consider $\varepsilon \in (0,2r]$ and set $y\coloneqq\gamma(r-\frac{\varepsilon}{2})\in B(p,r)$. Then $B(y,\frac{\varepsilon}{2})\subseteq B(q,\varepsilon)$, hence $B(y,\frac{\varepsilon}{2})\cap \partial B(p,r)=\emptyset$, i.e., $\partial B(p,r)$ is porous.
\end{proof}

\section{Proofs of some results about jumps and approximate differentiability points}\label{app:risultatitecnici}

\begin{proof}[Proof of Proposition \ref{jumpprop}] 
(i) We can without loss of generality assume that $k=1$. Consider a countable dense subset $\{(a_h,b_h,\nu_h): h\in\mathbb N\}$ of $\R\times\R\times\mathbb S^{m-1}$ and, for every $h\in \mathbb N$, define  $w_h:\R^n\to \R$ by
	\[
	w_h(y)\coloneqq\begin{cases}
	a_h& \text{ if } \widetilde L_{\nu_h}(y)\geq 0,\\
	b_h& \text{ if } \widetilde L_{\nu_h}(y)< 0.
	\end{cases}
	\]
We first prove that (recalling the notation \eqref{eq:defAr})
	\begin{equation}\label{eq:jumpborel}
	\left(\Omega\setminus\mathcal S_u\right)\cup \mathcal J_u=\bigcap_{\ell=1}^\infty\bigcup_{h=0}^\infty\left\{p \in \Omega: \limsup_{r\to 0}\fint_{A(r)}|u\circ F_p-w_h|d\mathscr L^n<\frac 1\ell\right\}.
	\end{equation}
The inclusion $\subseteq$ in \eqref{eq:jumpborel} is straightforward by Remark \ref{rem:applimIFFconvL1} and Proposition \ref{jumpequiv}. In order to prove the opposite inclusion, consider $p\in \Omega$ such that for every $\ell\in \mathbb N\setminus\{0\}$ there exists $w_{h_\ell}$ such that 
	\[
	\limsup_{r\to 0}\fint_{A(r)}|u\circ F_p-w_{h_\ell}|d\mathscr L^n<\frac 1\ell.
	\]
	We prove that, possibly passing to a subsequence, there exist $a, b$ and $\nu$ such that $(w_{h_\ell})$ is convergent in $L^1(A(1))$ to
        \[
        w(y)\coloneqq \begin{cases}
          a & \text{ if } \widetilde L_{\nu}(y)\geq 0,\\
          b & \text{ if } \widetilde L_{\nu}(y)<0.
     \end{cases}
        \]
Up to  subsequences we can suppose that  $(\nu_{h_\ell})$ converges to some $\nu$. Define $C\coloneqq\mathscr L^n\left(A(1)\right)$ and let $\overline \ell \in \mathbb N$ be such that for every $\ell,k\geq \overline \ell$ the set 
	\[
	A^+(1)\coloneqq \left\{y\in A(1):\widetilde L_{\nu_{h_\ell}}(y)>0\text{ and }\widetilde L_{\nu_{h_k}}(y)>0\right\}
	\]
        is such that $\mathscr L^n(A^+(1))\geq \tfrac 14 C$. By a change of variables, for such $h$ and $k$ one has
	\[
	\begin{aligned}
	|a_{h_\ell}-a_{h_k}|&=\fint_{A^+(1)}|a_{h_\ell}-a_{h_k}|d\mathscr L^n\leq\frac{4}{C} \int_{A^+(1)} |w_{h_\ell}-w_{h_k}|d\mathscr L^n\\
	&\leq\frac{4}{C}\int_{A(1)}|w_{h_\ell}-w_{h_k}|d\mathscr L^n=\frac{4}{Cr^Q}\int_{A(r)}|w_{h_\ell}-w_{h_k}|d\mathscr L^n\\
	&\leq 4\fint_{A(r)}|u\circ F_p-w_{h_\ell}|d\mathscr L^n+4\fint_{A(r)}|u\circ F_p-w_{h_k}|d\mathscr L^n.
	\end{aligned}
	\]
	Passing to the $\limsup$ as $r\to 0$ we get that $(a_{h_\ell})$ is Cauchy and therefore convergent to some $a\in \R$. Using the same technique we also get that $(b_{h_\ell})$ is convergent to some $b\in \R$, and $w_{h_\ell}$ converges in $L^1(A(1))$ to $w$. Now, for sufficiently large $\ell\in\mathbb N$ and for sufficiently small $r>0$, from
	\[
	\int_{A(r)}|u\circ F_p\circ \delta_r-w|d\mathscr L^n\leq \int_{A(r)}|u\circ F_p\circ \delta_r- w_{h_\ell}|d\mathscr L^n + \int_{A(1)}|w_{h_\ell}-w|d\mathscr L^n
	\]
we deduce the remaining inclusion $\supseteq$ in \eqref{eq:jumpborel}. 

Notice that the right-hand side of \eqref{eq:jumpborel} is a Borel set if, for any $h\in \mathbb N$, and any small enough $r$, the function
\begin{equation}\label{eq:fieldsmedal}
	p\longmapsto \fint_{A(r)}|u\circ F_p-w_h|d\mathscr L^n
\end{equation}
is continuous. This is clearly true if $u$ is of class $C^\infty$. For general $u$, fix $p\in \Omega$, $r>0$ and  $\varepsilon>0$ and consider $v\in C^\infty(\Omega)$ such that $ \|u-v\|_{L^1(B(p,C_1r))}<\varepsilon$, where $C_1$ is such that $F_p(A(r))\Subset B(p,C_1r)$. Applying the triangular inequality,  we find
	\[
	\begin{aligned}
	  &\fint_{A(r)}|u\circ F_p-u\circ F_q|d\mathscr L^n\\
	  \leq\: & \fint_{A(r)}\big(|u\circ F_p-v\circ F_p|
	+|v\circ F_p-v\circ F_q|
	+|v\circ F_q-u\circ F_q|\big)d\mathscr L^n\ <\ C \varepsilon,
	\end{aligned}
	\]
for some $C>0$, for every sufficiently small $r$ and for every $q$ sufficiently close to $p$. This proves that the function in \eqref{eq:fieldsmedal} is continuous. It follows that $\left(\Omega\setminus\mathcal S_u\right)\cup \mathcal J_u$ is a Borel set: then, also $\mathcal J_u$ is a Borel set, for $\Omega\setminus \mathcal S_u$ is  Borel  and it is disjoint from $\mathcal J_u$.

	Select now for any $p\in \mathcal J_u$ an $X$-jump triple $(u^+(p),u^-(p),\nu(p))$ according to Definition \ref{approximatejump}. Define  $\phi:\mathcal J_u\rightarrow\R^m$ by $\phi(p)\coloneqq(u^+(p)-u^-(p))\nu(p)$.  We prove that $\phi$ is Borel, so that also $\nu$ is Borel up to re-defining it as $\nu(p)=\phi(p)/|\phi(p)|$. Set
	\[
	w_p(y)\coloneqq\begin{cases}
	u^+(p)&\text{ if }\widetilde L_{\nu(p)}(y)>0;\\
	u^-(p)&\text{ if }\widetilde L_{\nu(p)}(y)<0,
	\end{cases}
	\]
and 
	\[
	\widetilde{A}(r)\coloneqq\left\{y\in \R^n:|(y_1,\dots,y_m)|+\sum_{j=m+1}^n|y_j|^{\frac{1}{w_j}}\leq r\right\}.
	\]
	Notice that the sets $\widetilde A(r)$ are  invariant under rotations of the first $m$ coordinates. By Proposition \ref{jumpequiv} we have that for every $\psi \in C_c^\infty(\widetilde A(1))$ and  every $i=1,\dots,n$
	\[
	\begin{aligned}
	\int_{\widetilde{A}(1)}w_p\partial_i\psi d\mathscr L^n&=\lim_{\varepsilon\to 0}\int_{\widetilde{A}(1)}(u\circ F_p\circ\delta_\varepsilon)\partial_i \psi d\mathscr L^n\\
	&=\lim_{\varepsilon\to 0}\frac{1}{\varepsilon^Q}\int_{\widetilde{A}(\varepsilon)}u(F_p(y))\partial_i\psi(\delta_{\varepsilon^{-1}}(y))d\mathscr L^n(y).
	\end{aligned}
	\]
	Hence, for every $\psi\in C_c^\infty(\widetilde A(1))$ and for every $i=1,\dots,n$ the function
	\[
	p\longmapsto\int_{\widetilde{A}(1)}w_p \partial_i\psi d\mathscr L^n
	\]
is Borel. Fix $p\in\mathcal J_u$ and consider a sequence $(\psi_h)$ in  $C_c^\infty(\widetilde A(1))$ converging to $\chi_{\widetilde A(1)}$. Computing the (Euclidean) measure derivative of $w_p$ we obtain that for every $i=1,\dots,n$ 
	\[
	\begin{aligned}
	&\phi^i(p)\mathscr H_e^{n-1}\big(\widetilde A(1)\cap\{\widetilde L_{\nu(p)}=0\}\big)\\
	=\: &D^iw_p(\widetilde A(1))=\lim_h\int_{\widetilde A(1)}\psi_hd D^iw_p=-\lim_h\int_{\widetilde A(1)}w_p\partial_i\psi_h d\mathscr L^n.
	\end{aligned}
	\]
	Since $\mathscr H_e^{n-1}(\widetilde A(1)\cap\{\widetilde L_{\nu(p)}=0\})$ does not depend on $p$ we deduce by the previous step that $\phi$ is a Borel function, and therefore $\nu$ is Borel. 
	
	Finally, by Proposition \ref{jumpequiv} we have
	\[
	u^+(p)=\lim_{\varepsilon\to 0}\frac{1}{\varepsilon^Q}\int_{A(\varepsilon)}\chi_{\{\widetilde L_{\nu(p)} >0\}}u\circ F_pd\mathscr L^n
	\]
and this concludes the proof.

The proof of (ii) is completely analogous to the Euclidean case, see \cite{AFP}.
\end{proof}

\begin{proof}[Proof of Proposition \ref{prop:DuBorel}]
We can assume without loss of generality that $k=1$. Consider  a dense subset $\{z_i:i\in\mathbb N\}$ of $\R^{m}$. Reasoning as in the proof of Proposition \ref{jumpprop} one can prove that 
	\[
	\mathcal D_u=\bigcap_{h=1}^\infty\bigcup_{i=0}^\infty\left\{p\in \Omega\setminus\mathcal S_u: \limsup_{\varrho\to 0}\frac{1}{r^{Q+1}}\int_{A(r)}\left|u\circ F_p-u^\star(p)-\widetilde L_{z_i}\right|d\mathscr L^n<\frac 1h\right\}
	\] 
which implies that $\mathcal D_u$ is a Borel set. \newline
We now prove that $D_X^{ap}u$ is Borel. Using Theorem \ref{normestimate}, for any $p\in \mathcal D_u$ one has
	\[
	\lim_{\varepsilon\to 0}\frac{1}{\varepsilon^{Q+1}}\int_{\delta_\varepsilon P}\left|u\circ F_p-u^\star(p)-\widetilde L_{D_X^{ap}u(p)}\right|d\mathscr L^n=0,
	\]
where for every $n$-tuple of positive real numbers $(\ell_1,\dots,\ell_n)$
	\[
	P=P(\ell_1,\dots,\ell_n)\coloneqq\{\xi\in\R^n: 0\leq\xi_j^{1/{d_j}}\leq \ell_j\  \text{for any }j=1,\dots,n\}
	\]
is the anisotropic box with axis that are parallel to the coordinate ones $(e_1,\dots,e_n)$. By a change of variables we get
	\[
        \frac{1}{\mathscr L^n(P)}\int_P \widetilde L_{D_X^{ap}u(p)} d\mathscr L^n=\frac{1}{\mathscr L^n(P)}\lim_{\varepsilon\to 0}\frac{1}{\varepsilon^{Q+1}}\int_{\delta_\varepsilon P} \left(u\circ F_p-u^\star(p)\right)d\mathscr L^n.
	\]
	From this we deduce that, for any $n$-tuple $(\ell_1,\dots,\ell_n)$ the function
	\begin{equation}\label{rectangle}
	p\longmapsto \frac{1}{\mathscr L^n(P)}\int_P\widetilde L_{D_X^{ap}u(p)} d\mathscr L^n
	\end{equation}
is Borel. Now, for every $i=1,\dots,m$ and  every $h\in\mathbb N\setminus\{0\}$ define the rectangles $P_h^i\coloneqq P(1/h,\dots,1/h,1,1/h,\dots,1/h)$. A simple computation shows that
	\[
	\lim_h \frac{1}{\mathscr L^n(P_h^i)}\int_{P_h^i}\widetilde L_{D_X^{ap}u(p)}d\mathscr L^n=\frac{1}{2}\left(D_X^{ap}u(p)\right)_i,
	\]
	which completes the proof.
\end{proof}

\begin{proof}[Proof of Proposition \ref{prop:localita(balneare?)}]
We can assume without loss of generality that $k=1$. 

(i) By Remark \ref{rem:applimIFFconvL1}, the functions $\widetilde u_r\coloneqq u\circ F_p\circ \delta_r$ and $\widetilde v_r\coloneqq u\circ F_p\circ \delta_r$ converge, respectively, to $u^\star(p)$ and $v^\star(p)$ in $L^1_{loc}(\R^n)$ as $r\to 0$. In particular, as $r\to 0$ the families $(\widetilde u_r)$ and $(\widetilde v_r)$ converge (locally) in measure  to $u^\star(p)$ and $v^\star(p)$ respectively. By a change of variables we have for any $R>0$
\[
\lim_{r\to 0}\mathscr L^n(\widehat B(0,R)\cap\{\widetilde v_r\neq \widetilde u_r\})
= \lim_{r\to 0}{r^{-Q}}\mathscr L^n(\widehat B(0,r R)\cap\{u\circ F_p\neq v\circ F_p \})=0.
\]
It follows that $(\widetilde u_r)$ and $(\widetilde v_r)$  have the same measure limit, hence $u^\star(p)= v^\star(p)$.

(ii) Using Proposition \ref{jumpequiv} and the same argument used in (i) we  obtain that the functions
	\[
	U(y)\coloneqq\begin{cases}
	u^+(p)&\text{ if }\widetilde L_{\nu_u(p)}(y)>0\\
	u^-(p)&\text{ if }\widetilde L_{\nu_u(p)}(y)<0
	\end{cases}
	\qquad\text{and}\qquad
	V(y)\coloneqq\begin{cases}
	v^+(p)&\text{ if }\widetilde L_{\nu_v(p)}(y)>0\\
	v^-(p)&\text{ if }\widetilde L_{\nu_v(p)}(y)<0
	\end{cases}
	\]
coincide for $\mathscr L^n$-almost every $y$, hence $(u^+(p),u^-(p),\nu_u(p))\equiv(v^+(p),v^-(p),\nu_v(p))$.

(iii) By point (i) we already know that $u^\star(p)=v^\star(p)$. Since
	\[
	\frac{u(F_p(\delta_r(y)))-u^\star(p)}{r}\neq\frac{v(F_p(\delta_r(y)))-v^\star(p)}{r}\quad\Longleftrightarrow\quad u(F_p(\delta_r(y)))\neq v(F_p(\delta_r(y))),
	\]
the statement follows using Proposition \ref{diffequiv} and an argument similar to part (i) above.
\end{proof}

\bibliographystyle{acm}
\bibliography{BVfunctionsBib}
\end{document}